\documentclass[oneside, 12pt]{amsart}

\usepackage{amsfonts}
\usepackage{amssymb}
\usepackage{latexsym}
\usepackage{graphics}
\usepackage[2emode]{psfrag}
\usepackage{eufrak}
\usepackage{amsthm}
\usepackage{amsmath}
\usepackage[all]{xy}
\usepackage{graphics}
\def\N{\mathbb N}

\def\*  C{{*  \mathcal C}}

\newcommand{\Ker}{{\rm Ker}\,}

\newcommand{\Rad}{{\rm Rad}}
\newcommand{\ord}{{\rm ord}}
\long\def\alert#1{\smallskip{\hskip\parindent\vrule%
\vbox{\advance\hsize-2\parindent\hrule\smallskip\parindent.4\parindent%
\narrower\noindent#1\smallskip\hrule}\vrule\hfill}\smallskip}

\newcommand{\Cc}{\mathcal{C}}

\def\*  C{{}*  \hspace*  {-1pt}{\Cc}}

\def\text#1{{\rm {\rm #1}}}

\newtheorem{prop}{Proposition}[section]
\newtheorem{lemma}[prop]{Lemma}
\newtheorem{cor}[prop]{Corollary}
\newtheorem{theo}[prop]{Theorem}

\theoremstyle{definition}
\newtheorem{Def}[prop]{Definition}
\newtheorem{ex}[prop]{Example}

\newtheorem{rem}[prop]{Remark}
\newtheorem{opprob}[prop]{Open problem}

\title{State BL-algebras}

\author{Lavinia Corina Ciungu$^1$, Anatolij Dvure\v{c}enskij$^2$, Marek Hy\v{c}ko$^2$}
\thanks{The first author thanks SAIA for the fellowship in Slovakia, Summer
 2009, two others are thankful for the support by Center of
Excellence SAS -~Quantum Technologies~-,  ERDF OP R\&D Project CE
QUTE ITMS  26240120009, the grants VEGA Nos. 2/0032/09, 2/7142/27
SAV and by the Slovak Research and Development Agency under the
contract No. APVV-0071-06, Bratislava.}

\begin{document}
\begin{abstract}
The concept of a state MV-algebra was firstly introduced by Flaminio
and Montagna in \cite{FlMo0} and \cite{FlMo} as an MV-algebra with
internal state as a unary operation. Di Nola and Dvure\v{c}enskij
gave a stronger version of a state MV-algebra in \cite{DiDV1},
\cite{DiDV2}. In the present paper we introduce the notion of a
state BL-algebra, or more precisely, a BL-algebra with internal
state.  We present different types of state  BL-algebras, like
strong state BL-algebras and state-morphism BL-algebras,  and we
study some classes of state BL-algebras. In addition, we give a sample of important examples of state BL-algebras and present some open problems.
\end{abstract}

\maketitle

\section{Introduction}

BL-algebras were introduced in Nineties by P. H\'ajek  as the
equivalent algebraic semantics for its basic fuzzy logic (for a
wonderful  trip through fuzzy logic realm,  see \cite{Haj1}). They
generalize theory of MV-algebras that is the algebraic semantics of
\L ukasiewicz many valued logic that was introduced in Fifties by
C.C. Chang \cite{Cha}. 40 years after appearing BL-algebras, D.
Mundici \cite{Mu} presented an analogue of probability, called a
state, as averaging process for formulas in \L ukasiewicz logic. In
the last decade, theory of states on MV-algebras and relative
structures is intensively studied by many authors, e.g. \cite{KuMu,
Kro, DvRa, DvuRac, Geo1, Rie} and others.

A new approach to states on MV-algebras was presented by T. Flaminio
and F. Montagna in \cite{FlMo0} and \cite{FlMo}; they added a unary
operation, $\sigma,$ (called as an inner state or a state-operator)
to the language of MV-algebras, which preserves the usual properties
of states.  It presents a unified approach to states and
probabilistic many valued logic in a logical and algebraic settings.
For example, H\'ajek's approach, \cite{Haj1}, to fuzzy logic with
modality $\mbox{Pr}$ (interpreted as {\it probably}) has the
following semantic interpretation: The probability of an event $a$
is presented as the truth  value of $\mbox{Pr}(a)$.

A. Di Nola and A. Dvure\v censkij gave in \cite{DiDV1} a stronger
version of state MV-algebras, namely state-morphism MV-algebras. In
particular, they completely described subdirectly irreducible
state-morphism MV-algebras. Such a description of only state
MV-algebras is yet unknown \cite{FlMo0, FlMo}.  And in \cite{DiDV2},
they described some types of state-morphism MV-algebras. In the
paper \cite{DDL}, the authors studied some subvarieties of state
MV-algebras, and they showed that any state MV-algebra whose
MV-reduct belongs to the variety $MV_n$ of MV-algebras  generated by
simple MV-chains $S_1,\ldots,S_n$ ($n\ge 1$), is always a
state-morphism MV-algebra.

In the present paper, we extend the definitions of state MV-algebras
and state-morphism MV-algebras to the case of BL-algebras and we
generalize the properties of the state-operator to this case.
Besides state-operators, we define strong state operators that in
the case of MV-algebras are identical with state-operators, and also
morphism-state-operators as state-operators preserving $\odot.$ To
illustrate these notions, we present some important examples of
state BL-algebras. We also study some classes of state-morphism
MV-algebras such as simple, semisimple, perfect and local
state-morphism MV-algebras, using the radical under a
state-morphism-operator and its properties. We show that under some
conditions, states and extremal states on the image of the
state-operator are in a one-to-one correspondence to states and
extremal states on the associated BL-algebra.

The paper is organized as follows: Section 2 recalls basic notions
and some results of BL-algebras and their classes which will be used
later in the paper. In Section 3 we define a state-operator and a
strong state-operator for a BL-algebra and prove some of their basic
properties. Section 4 gives   a  sample of important illustrative
examples, including BL-algebras corresponding some basic continuous
t-norms (like \L ukasiewicz, G\"odel and product). We show that if a
BL-algebra is linear, then every state-operator $\sigma$ is
necessarily an endomorphism such that $\sigma^2 = \sigma.$ Section 5
deals with state-filters and congruences. We show that subdirectly
irreducible state BL-algebras are not necessarily linear, and we
show some properties of radicals. In Section 6 we present  relations between states on BL-algebras and state-operators.  Finally, in Section 7 different
classes of state-morphism BL-algebras are presented, such as simple,
semisimple, perfect and local state-morphism BL-algebras. In
addition, we present some open problems.

\section{Elements of BL-algebras}

In the present section, we gather  basic definitions and properties
on BL-algebras for reader's convenience.

\begin{Def}(\cite{Haj1})\label{def:BL}
A BL-algebra is an algebra $(A, \wedge, \vee, \odot, \rightarrow,
0,1)$ of the type $\langle 2,2,2,2,0,0\rangle$ such that $(A,
\wedge,\vee,0,1)$ is a bounded lattice, $(A,\odot,1)$ is a
commutative monoid,
and for all $a,b,c \in A,$\\
$(1)$ $c\leq a\rightarrow b$ iff $a\odot c\leq b;$\\
$(2)$ $a\wedge b=a\odot (a\rightarrow b);$\\
$(3)$ $(a\rightarrow b)\vee (b\rightarrow a)=1.$
\end{Def}

Let $a \in A$, we set $a^0:=0$ and $a^n :=a^{n-1}$ for any integer
$n\ge 1.$  If there is the least integer $n$ such that $a^n = 0,$ we
set $\ord(a) = n,$  if there is no such an integer, we set
$\ord(a)=\infty.$

The following well-known properties of BL-algebras will be used in
the sequel.

\begin{prop}\label{4}
Let $A$ be a BL-algebra. Then\\
$(1)$ if $a\leq b$ and $c\leq d$ then $a\odot c\leq b\odot d;$ \\
$(2)$ if $a\leq b$ then $c\rightarrow a\leq c\rightarrow b;$\\
$(3)$ $a\rightarrow b^-=(a\odot b)^-;$\\
$(4)$ $a\rightarrow a\wedge b = a\rightarrow b;$\\
$(5)$ $a\to b \le a\odot c \to b\odot c;$\\
$(6)$ $a\to (b\to c) = (a\odot b)\to a.$
\end{prop}

We define the following operations known in any BL-algebra $A$:\\
$x\oplus y:=(x^-\odot y^-)^-,$ $x\ominus y:=x\odot y^-$ and
$d(x,y):=(x\rightarrow y)\odot (y\rightarrow x)$ for any $x,y \in
A$.

We recall a few definitions of states that we will use in the next
sections. We note that a state is an analogue of averaging process
for formulas in \L ukasiewicz logic \cite{Mu} or in fuzzy logic
\cite{Geo1, Rie}.

According to \cite{Geo1}, we say that
a {\it Bosbach state} on $A$ is a function $s:A\rightarrow [0,1]$ such that the following conditions hold:\\
$(BS1)$ $s(x)+s(x\rightarrow y)=s(y)+s(y\rightarrow x), $ $x,y \in A;$\\
$(BS2)$ $s(0)=0$ and $s(1)=1.$

For another notion of a state given in \cite{Rie}, we introduce a
partial relation $\perp$ as follows: We say that two elements $x,y
\in A$ are said to be {\it orthogonal} and we write $x\perp y$ if
$x^{--}\leq y^-.$ It is simple to show that $x\perp y$ iff $x\le
y^-$ and iff $x\odot y=0.$ It is clear that $x\perp y$ iff $y \perp
x,$ and $x\perp 0$ for each $x \in A.$

For two orthogonal elements $x,y$ we define a partial binary
operation, $+,$ on $A$ via $x+y:=y^-\rightarrow
x^{--}(=x^-\rightarrow y^{--}).$

A function $s:A\rightarrow [0,1]$ is called a {\it Rie\v can state} if the following conditions hold:\\
$(RS1)$ if $x\bot y,$ then $s(x+y)=s(x)+s(y);$\\
$(RS2)$ $s(0)=0.$

As it was shown in \cite{DvRa}, every Bosbach state is a Rie\v{c}an
state and vice versa, therefore for the rest of the paper a Bosbach
state or a Rie\v{c}an state will be  shortly called a {\it  state}.
We denote by ${\mathcal S}(A)$ the set of all states on $A.$ We
recall that ${\mathcal S}(A)$ is always non-void, see Remark
\ref{re:maxim} below.

We remind  that a net of states, $\{s_\alpha\},$ {\it converges
weakly} to a state, $s,$ if $\lim_\alpha s_\alpha(x)=s(x)$ for every
$x \in A.$ In addition, if $s$ is a state, then $\Ker(s):=\{x\in
A\mid s(x)=1\}$ is a filter.

A {\it state-morphism} on $A$ is a function $m:A\rightarrow [0,1]$ satisfying:\\
$(SM1)$ $m(0)=0;$\\
$(SM2)$ $m(x\rightarrow y)=\min\{1-m(x)+m(y),1\},$\\
for any $x,y$ in $A.$

We note that by \cite{Geo1}, every state-morphism is a state.

A state $s$ on $A$ is called an {\it extremal state} if for any
$0<\lambda<1$ and for any two states $s_1, s_2$ on $A,$ $s=\lambda
s_1+(1-\lambda) s_2$ implies $s_1=s_2=s.$  By $\partial_e{\mathcal
S}(A)$ we denote the set of all extremal states. Due to the
Krein--Mil'man theorem, \cite[Thm 5.17]{Goo}, every  state on $A$ is
a weak limit of  a net of convex combinations of extremal states.

\begin{theo}$\rm($\cite{Geo1, DvRa}$\rm)$\label{13}
Let $m:A\rightarrow [0,1]$ be a state. Then $m$ is an extremal state
iff $m(x\vee y)=\max\{m(x), m(y)\}$ for any $x,y$ in $A$ iff
$m(x\odot y)=m(x)\odot m(y):= \min\{m(x)+m(y)-1,0\}$ for any $x,y$
in $A$ iff $m$ is a state-morphism iff $\Ker(m)$ is a maximal
filter.
\end{theo}

We remind a few definitions and results related to the  notion of a
{\it filter.}

A non-empty set $F\subseteq A$ is called a {\it filter} of $A$ (or a {\it BL-filter} of $A$) if for every $x,y\in A$:\\
$(1)$ $x,y\in F$ implies $x\odot y\in F;$\\
$(2)$ $x\in F,$ $x\leq y$ implies $y\in F.$

A proper filter $F$  of $A$ is called a {\it maximal filter} if it
is not strictly contained in any other proper filter.

We denote by $\Rad(A)$ the intersection of all maximal filters of
$A.$

\begin{prop}$\rm($\cite{DiGeIo}$\rm)$\label{27}
If $F$ is a proper filter in a nontrivial BL-algebra $A,$ then the following are equivalent:\\
$(1)$ $F$ is maximal;\\
$(2)$ for any $x\in A,$ $x\notin F$ implies $(x^n)^- \in F$ for some
$n\in \N.$
\end{prop}

A BL-algebra is called {\it local} if it has a unique maximal filter.

A proper filter $P$ of $A$ is called {\it
primary} if for  $a,b\in A,$  $(a\odot b)^-\in P$ implies
$(a^n)^-\in P$ or $(b^n)^-\in P$ for some $n\in \N.$

\begin{prop}$\rm($\cite{Leo}$\rm)$\label{27.1}
In a BL-algebra $A$ the following are equivalent:\\
$(1)$ $A$ is local;\\
$(2)$ any proper filter of $A$ is primary.
\end{prop}

\begin{prop}$\rm($\cite{Tur}$\rm)$\label{27'''}
A BL-algebra is local iff for any $x\in A,$ $\ord(x)<\infty$ or
$\ord(x^-)<\infty.$
\end{prop}

\begin{rem}$\rm($\cite{Tu}$\rm)$
If $F$ is a filter of a BL-algebra $A,$ then we define the
equivalence relationship $x\sim_F y$ iff $(x\rightarrow y)\odot
(y\rightarrow x)\in F.$ Then $\sim_F$ is a congruence and  the
quotient algebra $A/F$ becomes a BL-algebra with the natural
operations induced from those on $A.$ Denoting by $x/F$ the
equivalence class of $x$, then $x/F = 1/F$ iff $x\in F.$ Conversely,
if $\sim$ is a congruence, then $F_\sim:=\{x\in A\mid x\sim 1\}$ is
a filter, and $\sim_{F_\sim}=\sim,$ and $F=F_{\sim_F}.$
\end{rem}

\begin{rem}\label{re:maxim}
If $F$ is a maximal filter on a BL-algebra $A,$ then $A/F$ is always
isomorphic to a subalgebra of the real interval $[0,1]$ that is
simultaneously an MV-algebra as well as a BL-algebra such that
$x^{--}=x$ for all $x \in [0,1].$ In other words, the mapping
$x\mapsto x/F,$ $x \in A,$ is a state-morphism.
\end{rem}

\begin{prop}$\rm($\cite{GeoLeo, Tur}$\rm)$\label{0}
A filter $P$ of $A$ is primary if and only if $A/P$ is local.
\end{prop}

\begin{prop}$\rm($\cite[Cor. 1.16]{DiGeIo}$\rm)$\label{27''}
$\Rad(A)=\{x\in A\mid (x^n)^-\leq x, \forall \ n\in \N\}.$
\end{prop}

Denote $\Rad(A)^- = \{x^- \mid x\in \Rad(A)\}.$ The element $x\in A$
such that $(x^n)^-\le x$ for every integer $n \ge 1$ is said to be
{\it co-infinitesimal}. The latter proposition says that $\Rad(A)$
consists only from all co-infinitesimal elements of $A.$

\begin{rem}\label{26'}
One can easily check that if $x\in \Rad(A),$ then $x^-\in \Rad(A)^-$
and if $x\in \Rad(A)^-$, then $x^-\in \Rad(A).$
\end{rem}

A BL-algebra $A$ is called {\it perfect} if, for any $x \in A,$
either $x\in \Rad(A)$ or $x\in \Rad(A)^-.$

\begin{cor} $\rm($\cite{Ciu}$\rm)$\label{26''}
In a perfect BL-algebra $A$, if $x\in \Rad(A)$ and $y\in \Rad(A)^-$,
then $x^-\leq y^-.$
\end{cor}

A BL-algebra $A$ is called $(1)$ {\it simple} if $A$ has two
filters, $(2)$ {\it semisimple} if $\Rad(A)=\{1\},$ and $(3)$ {\it
locally finite} if for any $x\in A,$ $x\neq 1,$ $\ord(x)<\infty.$

\begin{lemma} \label{100} If $A$ is a BL-algebra, the following are equivalent:\\
$(1)$ $A$ is locally finite;\\
$(2)$ $A$ is simple.

In such a case, $A$ is linearly ordered.
\end{lemma}
\begin{proof}
First, assume $A$ is locally finite. Consider $F$ a proper filter of
$A$ and let $x\in F\subseteq A, x\neq 1$. There exists $n\in \N,
n\geq 1$ such that $x^n=0$, so $0\in F$ which is a contradiction.
Thus the only proper filter of $A$ is $\{1\}$, that is $A$ is
simple.

Now, suppose $A$ is simple and let $x\in A, x\neq 1$ such that
$\ord(x)=\infty$. Then the filter $F(x)$ generated by $x$ is proper
and $F(a)\neq \{1\}$ which contradicts the hypothesis. It follows
that $A$ is locally finite.

The linearity of $A$ follows from \cite[Prop 2.14]{DiGeIo}.
\end{proof}

\section{State BL-algebras}

Inspired by T. Flaminio and F. Montagna \cite{FlMo0,FlMo}, we
enlarge the language of BL-algebras by introducing a new operator,
an internal state.

\begin{Def}\label{2'} Let $A$ be a BL-algebra. A mapping $\sigma:A
\to A$ such that, for all $x,y \in A,$ we have\\
$(1)_{BL}$ $\sigma(0)=0;$\\
$(2)_{BL}$ $\sigma(x\rightarrow y)=\sigma(x)\rightarrow \sigma(x\wedge y);$\\
$(3)_{BL}$ $\sigma(x\odot y)=\sigma(x)\odot \sigma(x\rightarrow x\odot y);$\\
$(4)_{BL}$ $\sigma(\sigma(x)\odot \sigma(y))=\sigma(x)\odot \sigma(y);$\\
$(5)_{BL}$ $\sigma(\sigma(x)\rightarrow
\sigma(y))=\sigma(x)\rightarrow \sigma(y)$ \\
is said to be a {\it state-operator} on $A,$ and the pair
$(A,\sigma)$ is said to be a {\it state BL-algebra}, or more
precisely, a {\it BL-algebra with  internal state}.
\end{Def}

We recall that the class of state BL-algebras forms a variety.

If $\sigma$ is a state-operator, then $\Ker(\sigma):=\{x\in A \mid
\sigma(x)=1\}$ is said to be the {\it kernel} of $\sigma$ and it is a filter
(more precisely a state-filter, see Section \ref{sec5}). A
state-operator $\sigma$ is said to be {\it faithful} if
$\Ker(\sigma)=\{1\}.$

\begin{ex}\label{ex:1}
$(A, \mbox{\rm id}_A)$ is a state BL-algebra.
\end{ex}

\begin{ex}\label{ex:2}  Let $A$ be a  BL-algebra. On $A\times
A$ we define two operators, $\sigma_1$ and $\sigma_2,$ as follows
$$
\sigma_1(a,b)=(a,a),\quad \sigma_2(a,b)=(b,b),\quad (a,b)\in A\times
A.\eqno(3.1)
$$
Then $\sigma_1$ and $\sigma_2$ are two state-operators on $A\times
A$ that are also  endomorphisms such that $\sigma_i^2 =\sigma_i,$
$i=1,2.$ Moreover, $(A\times A,\sigma_1)$ and $(A\times A,\sigma_2)$
are isomorphic state BL-algebras under the isomorphism $(a,b)\mapsto
(b,a).$
\end{ex}

\begin{ex}\label{4elemEx}
Consider $A=\{0,a,b,1\}$ where $0<a<b<1$. Then
$(A,\wedge,\vee,\odot,\rightarrow,0,1)$ is a BL-algebra that is not
an MV-algebra (\cite{Tur}) with the operations:
\begin{center}
\begin{tabular}{|c|c|c|c|c|c|}
\hline
$\mathbf{\odot}$ & \textbf{0} & \textbf{a} & \textbf{b} & \textbf{1} \\ \hline
\textbf{0} & 0 & 0 & 0 & 0 \\
\hline
\textbf{a} & 0 & 0 & a & a \\
\hline
\textbf{b} & 0 & a & b & b \\
\hline
\textbf{1} & 0 & a & b & 1 \\
\hline
\end{tabular}
\qquad
\begin{tabular}{|c|c|c|c|c|c|}
\hline
$\mathbf{\rightarrow}$ & \textbf{0} & \textbf{a} & \textbf{b} & \textbf{1} \\
\hline
\textbf{0} & 1 & 1 & 1 & 1 \\
\hline
\textbf{a} & a & 1 & 1 & 1 \\
\hline
\textbf{b} & 0 & a & 1 & 1 \\
\hline
\textbf{1} & 0 & a & b & 1 \\
\hline
\end{tabular}
\end{center}

The operation $\oplus$ is given by the table:
\begin{center}
\begin{tabular}{|c|c|c|c|c|c|}
\hline
$\mathbf{\oplus}$ & \textbf{0} & \textbf{a} & \textbf{b} & \textbf{1} \\
\hline
\textbf{0} & 0 & a & 1 & 1 \\
\hline
\textbf{a} & a & 1 & 1 & 1 \\
\hline
\textbf{b} & 1 & 1 & 1 & 1 \\
\hline
\textbf{1} & 1 & 1 & 1 & 1 \\
\hline
\end{tabular}
\end{center}

One can easily check that the unary operation $\sigma$ defined as follows:\\
$\hspace*{3cm}$ $\sigma(0)=0,\ \sigma(a)=a,\ \sigma(b)=1,\ \sigma(1)=1$ \\
is a state-operator on $A$. Therefore, $(A,\sigma)$ is a state
BL-algebra. Moreover, the following identities hold: $\sigma(x\odot
y)=\sigma(x)\odot \sigma(y)$ and $\sigma(x\rightarrow
y)=\sigma(x)\rightarrow \sigma(y)$ for all $x,y\in A$, so $\sigma$
is a BL-endomorphism and $\sigma(A)=\{0,a,1\}.$

\end{ex}

\begin{lemma}\label{3new}
In a state BL-algebra $(A,\sigma)$ the following hold:\\
$(a)$ $\sigma(1)=1;$\\
$(b)$ $\sigma(x^-) = \sigma(x)^-;$\\
$(c)$ if $x\leq y,$ then $\sigma(x)\leq \sigma(y);$\\
$(d)$ $\sigma(x\odot y)\geq \sigma(x)\odot \sigma(y)$ and if $x\odot y = 0,$ then $\sigma (x\odot y)=\sigma(x)\odot \sigma(y);$\\
$(e)$ $\sigma(x\ominus y)\geq \sigma(x)\ominus \sigma(y)$ and if $x\leq y,$ then $\sigma(x\ominus y)=\sigma(x)\ominus \sigma(y);$\\
$(f)$ $\sigma(x\wedge y) = \sigma(x)\odot \sigma(x\rightarrow y)$;\\
$(g)$ $\sigma(x\rightarrow y)\leq \sigma(x)\rightarrow \sigma(y)$ and if $x, y$ are comparable, then $\sigma(x\rightarrow y) = \sigma(x)\rightarrow \sigma(y)$;\\
$(h)$ $\sigma(x\rightarrow y)\odot \sigma(y\rightarrow x)\leq d(\sigma(x),\sigma(y));$\\
$(i)$ $\sigma(x)\oplus \sigma(y)\geq \sigma(x\oplus y)$ and if $x\oplus y=1,$ then $\sigma(x)\oplus \sigma(y)=\sigma(x\oplus y)=1;$\\
$(j)$ $\sigma(\sigma(x))=\sigma(x);$\\
$(k)$ $\sigma(A)$ is a BL-subalgebra of $A;$\\
$(l)$ $\sigma(A)=\{x\in A: x=\sigma(x)\};$\\
$(m)$ if $\ord(x)< \infty,$ then $\ord(\sigma(x))\le \ord(x)$ and  $\sigma(x)\notin \Rad(A);$\\
$(n)$ $\sigma(x\rightarrow y) = \sigma(x)\rightarrow \sigma(y)$ iff
$\sigma(y\rightarrow x) = \sigma(y)\rightarrow \sigma(x);$\\
$(o)$ if $\sigma(A) = A,$ then $\sigma$ is the identity on $A;$\\
$(p)$ if $\sigma$ is faithful, then $x<y$ implies
$\sigma(x)<\sigma(y);$\\
$(q)$ if $\sigma$ is faithful then either $\sigma(x)=x$ or
$\sigma(x)$ and $x$ are not comparable;\\
$(r)$ if $A$ is linear and $\sigma$ faithful, then  $\sigma(x)=x$
for any $x \in A.$

\end{lemma}

\begin{proof}
$(a)$ $\sigma(1)=\sigma(x\rightarrow x)=\sigma(x)\rightarrow
\sigma(x\wedge x)=1$ using condition $(2)_{BL}.$

$(b)$ $\sigma(x^-)=\sigma(x\rightarrow 0)=\sigma(x)\rightarrow
\sigma(x\wedge 0)=\sigma(x)\rightarrow 0=\sigma(x)^-$ using
$(1)_{BL}$ and $(2)_{BL}.$

$(c)$ If $x\leq y,$ $x=y\odot (y\rightarrow x),$ so by $(3)_{BL}$,
we get
$\sigma(x)=\sigma(y\odot (y\rightarrow x))=\sigma(y)\odot \sigma(y\rightarrow (y\odot (y\rightarrow x))) \leq \sigma(y).$

$(d)$ By $(3)_{BL},$ (5) of Proposition \ref{4} and $(c)$ we have
$\sigma(x\odot y)=\sigma(x)\odot \sigma(x\rightarrow x\odot y)\geq
\sigma(x)\odot \sigma(y),$ because in view of Proposition
\ref{4}(5), $y\leq x\rightarrow x\odot y$. If $x\odot y = 0$, then
$y\leq x^-$, so $\sigma(x)\odot \sigma(y)\leq \sigma(x)\odot
\sigma(x^-) = \sigma(x)\odot \sigma(x)^- = 0 = \sigma(0) =
\sigma(x\odot y)$.

$(e)$ Since $x\ominus y:=x\odot y^-,$ so then $\sigma(x\ominus
y)=\sigma(x\odot y^-)=\sigma(x)\odot \sigma(x\rightarrow x\odot
y^-)\geq \sigma(x)\odot \sigma(y^-)=\sigma(x)\odot
\sigma(y)^-=\sigma(x)\ominus \sigma(y),$ using $(3)_{BL}$ and $(c).$
Now if $x\leq y,$ then $y^-\leq x^-$ and $x\odot y^- \leq x\odot x^-
= 0$, so by $(d)$ we obtain equality.

$(f)$ $\sigma(x\wedge y) = \sigma(x\odot (x\rightarrow y)) =
\sigma(x)\odot \sigma(x\rightarrow (x\odot (x\rightarrow y))) =
\sigma(x)\odot \sigma(x\rightarrow (x\wedge y)) = \sigma(x)\odot
\sigma(x\rightarrow y)$, using Proposition \ref{4}(4) .

$(g)$ Using $(2)_{BL}$, $(c)$ and Proposition \ref{4}(2), it follows
that $\sigma(x\rightarrow y)=\sigma(x)\rightarrow \sigma(x\wedge
y)\leq \sigma(x)\rightarrow \sigma(y).$ If $x\leq y$, then from
$(c)$ we get $\sigma(x)\leq \sigma(y)$ and $\sigma(x\rightarrow y) =
\sigma(1) = 1 = \sigma(x)\rightarrow \sigma(y)$. Let now $y\leq x$,
then $x\wedge y = y$ and from $(2)_{BL}$ we have the desired
equality.

$(h)$ From $(g)$ we know that $\sigma(x\rightarrow y)\leq
\sigma(x)\rightarrow \sigma(y)$ and similarly $\sigma(y\rightarrow
x)\leq \sigma(y)\rightarrow \sigma(x)$. Using Proposition \ref{4}(1)
we get $\sigma(x\rightarrow y)\odot \sigma(y\rightarrow x)\leq
d(\sigma(x),\sigma(y)).$

$(i)$ We know $x\oplus y=(x^-\odot y^-)^-.$
 By $(b)$ and $(d),$
$\sigma(x^-\odot y^-)\geq \sigma(x^-)\odot \sigma(y^-),$ so
$(\sigma(x^-)\odot \sigma(y^-))^-\geq \sigma((x^-\odot y^-))^-,$
which implies
$\sigma(x)\oplus \sigma(y)\geq \sigma(x\oplus y).$\\
If $x\oplus y=1,$ then $1=\sigma(x\oplus y)\leq \sigma(x)\oplus
\sigma(y),$ so $\sigma(x)\oplus \sigma(y)=1$ and thus
$\sigma(x)\oplus \sigma(y)=\sigma(x\oplus y).$

$(j)$ Replacing $y=1$ in $(4)_{BL}$ and using $(a)$ we get:
$\sigma(\sigma(x))=\sigma(\sigma(x)\odot \sigma(1))=\sigma(x)\odot
\sigma(1)=\sigma(x).$

$(k)$ From $(1)_{BL}, (4)_{BL}, (5)_{BL}$ and $(a)$ it follows that
$\sigma(A)$ is closed under all BL-operations $\odot, \to, \wedge$
and $\vee.$ Thus $\sigma(A)$ is a BL-subalgebra of $A$.

 $(l)$ For the direct inclusion, consider
$x\in \sigma(A),$ that is $x=\sigma(a)$ for some $a\in A.$ Then
$\sigma(x)=\sigma(\sigma(a))=\sigma(a)$ by $(j),$ and thus
$x=\sigma(x).$ The other inclusion is straightforward.

$(m)$  Let $m=\ord(x).$ By $(d),$ $0=\sigma(x^m)\ge \sigma(x)^m$
proving $\ord(\sigma(x))\le \ord(x).$ From Proposition \ref{27''} we
conclude any element of finite order cannot belong to $\Rad(A).$

$(n)$ Let $\sigma(x\rightarrow y) = \sigma(x)\rightarrow \sigma(y)$.
Then by $(f),$ $\sigma(y\rightarrow x) = \sigma(y)\rightarrow
\sigma(y\wedge x) = \sigma(y)\rightarrow (\sigma(x)\odot
\sigma(x\rightarrow y)) = \sigma(y)\rightarrow (\sigma(x)\odot
(\sigma(x)\rightarrow \sigma(y))) = \sigma(y)\rightarrow
\sigma(x)\wedge \sigma(y) = \sigma(y)\rightarrow \sigma(x)$. The
converse implication is proved by exchanging $x$ and $y$ in the
previous formulas.

$(o)$ For any $x \in A,$ we have $x =\sigma(x_0)$ for some $x_0\in
A.$ By $(j)$, we have $\sigma(x)=
\sigma(\sigma(x_0))=\sigma(x_0)=x.$

$(p)$ Suppose the converse, i.e. $\sigma(x)=\sigma(y).$  Then
$\sigma(y\to x)=\sigma(y)\to \sigma(x)=1$ giving $y\le x,$ absurd.

$(q)$ Let  $x$ be such that $\sigma(x)\ne x$ and let $x$ and
$\sigma(x)$ be comparable. Then $x<\sigma(x)$ or $\sigma(x)<x$
giving $\sigma(x)<\sigma(x),$ a contradiction.

$(r)$ It follows directly from $(q).$
\end{proof}

\begin{rem}\label{re:3.6}
It is interesting to note that for MV-algebras and linear product
BL-algebras \cite[Lemma 4.1]{CEGT} we have $x\rightarrow x\odot y =
x^- \vee y$. So for these subvarieties the axiom $(3)_{BL}$ can be
rewritten in the form:
$$ \sigma(x\odot y)=\sigma(x)\odot \sigma(x^-\vee y).\eqno(3')_{BL}$$
\end{rem}

\begin{Def}   A {\it strong state-operator} on a BL-algebra $A$ is a
mapping $\sigma:A\to A$ satisfying $(1)_{BL},(2)_{BL},(3')_{BL},
(4)_{BL},$ and $(5)_{BL}.$  The couple $(A,\sigma)$ is called a {\it
strong state BL-algebra}.
\end{Def}

In what follows, we show that a strong state-operator is always a
state-operator.

\begin{prop}\label{pr:3.8}
Every strong state BL-algebra is a state BL-algebra.
\end{prop}
\begin{proof}  Let $\sigma$ be a strong state-operator on $A.$
We prove that $(3')_{BL}$ implies $(3)_{BL}.$ Indeed,  we have
$\sigma(x\wedge y) = \sigma(x\odot (x\rightarrow y)) =
\sigma(x)\odot \sigma(x^-\vee (x\rightarrow y)) = \sigma(x)\odot
\sigma(x\rightarrow y)$. Replacing $y$ by $x\odot y$ in the previous
identity we obtain: $\sigma(x\wedge (x\odot y)) = \sigma(x\odot y) =
\sigma(x)\odot \sigma(x\rightarrow x\odot y)$, which is axiom
$(3)_{BL}$. Thus we obtain $(3')_{BL}\Rightarrow (3)_{BL}.$
\end{proof}

We recall that the converse implication is not known.

For strong state BL-algebras, the properties stated in Lemma
\ref{3new} can be extended as follows:

\begin{lemma}\label{3}
Let $(A,\sigma)$ be a strong state BL-algebra. Then, for all $x,y \in A,$ we have:\\
$(a)$ $\sigma(x\odot y)\geq \sigma(x)\odot \sigma(y)$ and if $x^-\leq y$ then $\sigma (x\odot y)=\sigma(x)\odot \sigma(y);$\\
$(b)$ $\sigma(x\ominus y)\geq \sigma(x)\ominus \sigma(y)$ and if $x$ and $y$ are comparable then $\sigma(x\ominus y)=\sigma(x)\ominus \sigma(y);$\\
$(c)$ $\sigma(x\odot \sigma(x ^-))= \sigma(x ^-\odot\sigma(x))$ for
any $x \in A.$

\end{lemma}

\begin{proof}
$(a)$ By $(3')_{BL}$ and $(c)$ we have $\sigma(x\odot
y)=\sigma(x)\odot \sigma(x^-\vee y)\geq \sigma(x)\odot \sigma(y).$
If $x^-\leq y,$ then $x^-\vee y=y,$ thus we obtain the desired
equality.

$(b)$ We know: $x\ominus y=x\odot y^-,$ so then $\sigma(x\ominus y)=\sigma(x\odot y^-)=\sigma(x)\odot \sigma(x^- \vee y^-)\geq \sigma(x)\odot \sigma(y^-)=\sigma(x)\odot \sigma(y)^-=\sigma(x)\ominus \sigma(y),$ using $(3')_{BL}$ and $(c).$\\
Now assume $x$ and $y$ are comparable. The case $x\leq y$ is proved
in Lemma \ref{3new}$(e)$. If $y\leq x$, then $x^-\leq y^-$ and
$x^-\vee y^- = y^-$, and thus we have equality.

$(c)$ Check $ \sigma(x \odot\sigma(x^-))=\sigma (x)\odot \sigma
(x^-\vee \sigma(x^-))= \sigma (x^-\odot  \sigma(x)).$
\end{proof}

\begin{lemma}\label{arrow}
Let $(A, \sigma)$ be a state BL-algebra. The following hold: \\
$(1)$ Let $x,y\in A$ be fixed. Then $\sigma(x\rightarrow y) =
\sigma(x)\rightarrow \sigma(y)$
if and only if $\sigma(x\wedge y) = \sigma(x)\wedge \sigma(y)$;\\
$(2)$ $\sigma(x\rightarrow y) = \sigma(x)\rightarrow \sigma(y)$ for
all $x,y \in A$ if and only if $\sigma(x\vee y) = \sigma(x)\vee
\sigma(y)$ for all $x,y \in A$.\\
$(3)$ If $\sigma(x\to y)=\sigma(x)\to \sigma(y)$ for all $x,y\in A,$
then $\sigma(x\odot y)=\sigma(x)\odot \sigma(y)$ for all $x,y \in
A,$ and $\sigma$ is an endomorphism.
\end{lemma}

\begin{proof}
(1)  Let $x,y\in A$ be fixed.
For the direct implication, we use Lemma \ref{3new}$(f)$ and we get $\sigma(x\wedge y) = \sigma(x)\odot \sigma(x\rightarrow y) = \sigma(x)\odot (\sigma(x)\rightarrow \sigma(y)) = \sigma(x)\wedge \sigma(y)$.\\
For the converse implication, we have: $\sigma(x\rightarrow y) =
\sigma(x)\rightarrow \sigma(x\wedge y) = \sigma(x)\rightarrow
(\sigma(x)\wedge \sigma(y)) = \sigma(x)\rightarrow \sigma(y)$, by
Lemma \ref{4}(4).

 (2) Assume $\sigma(x\to y) = \sigma(x)\to
\sigma(y)$ for all $x,y \in A.$ Using the identity $x\vee y =
[(x\rightarrow y)\rightarrow y] \wedge [(y\rightarrow x)\rightarrow
x],$  then from (1) we have that $\sigma$ preserves all meets in
$A.$ It is straightforward now that $\sigma(x\vee y) = \sigma(x)\vee
\sigma(y)$.

Conversely, assume that $\sigma$ preserves all meets in $A.$  Due to
the identity $(x\vee y)\to y = x\to y,$ we have $\sigma(x\to y) =
\sigma((x\vee y)\to y) = \sigma(x\vee y) \to \sigma((x\vee y)\wedge
y) = (\sigma(x)\vee \sigma(y)) \to \sigma(y) = \sigma(x)\to
\sigma(y).$

(3) $\sigma(x\odot y)\to \sigma(z)=\sigma(x\odot y \to z) =\sigma(x
\to(y\to z)) = \sigma(x)\to (\sigma(y) \to \sigma(z)) =
(\sigma(x)\odot \sigma(y)) \to \sigma(z).$

Hence, if $z=\sigma(x)\odot \sigma(y)$, we have $\sigma(x\odot y)
\to \sigma(\sigma(x)\odot \sigma(y))= \sigma(x)\odot \sigma(y) \to
\sigma(\sigma(x)\odot \sigma(y))= \sigma(x)\odot \sigma(y)\to
\sigma(x)\odot \sigma(y)=1.$ This yields $\sigma(x\odot y) \to
\sigma(\sigma(x)\odot \sigma(y))= \sigma(x\odot y) \to
\sigma(x)\odot \sigma(y)= 1$ and whence $\sigma(x\odot y) \le
\sigma(x)\odot \sigma(y).$ The converse inequality follows from
Lemma \ref{3new}$(d).$

\end{proof}

We note that from the proof of (2) of the previous proposition we
have that if $x,y \in A$ are fixed and $\sigma(x\vee y) =
\sigma(x)\vee \sigma(y),$ then $\sigma(x\to y) = \sigma(x)\to
\sigma(y).$

\begin{lemma}\label{linear}
Let $(A, \sigma)$ be a linearly ordered state BL-algebra. Then for $x, y\in A,$ we have:\\
$(1)$ $\sigma(x\rightarrow y) = \sigma(x)\rightarrow \sigma(y)$;\\
Moreover, if $(A, \sigma)$ is strong, we have:\\
$(2)$ $\sigma(x\odot y) = \sigma(x)\odot \sigma(y)$.

\end{lemma}

\begin{proof}
$(1)$ It is a direct consequence of $(g)$ in Lemma \ref{3new}.\\
$(2)$ Assume $x\leq y$. (The case $y\leq x$ is can be treated
similarly.) Then $y^-\leq x^-$ and we have the following two cases:
(I) $x^-\leq y$ and (II) $y\leq x^-$. Case (I) follows from
condition $(a)$ of Lemma \ref{3}. In case (II) we have $x\odot y\leq
x\odot x^- = 0$ and $\sigma(x^-) = \sigma(x)^-\geq \sigma(y)$. Thus
$\sigma(x\odot y) = \sigma(0) = 0 = \sigma(x)\odot \sigma(x)^-\geq
\sigma(x)\odot \sigma(y)$.
\end{proof}

\begin{Def}(\cite{FlMo0,FlMo})
A state MV-algebra is a pair $(M, \sigma)$ such that $(M,\oplus,\odot,^-,0,1)$ is
an MV-algebra and $\sigma$ is a unary operation on $M$ satisfying:\\
$(1)_{MV}$ $\sigma(1)=1;$\\
$(2)_{MV}$ $\sigma(x^-)=\sigma(x)^-;$\\
$(3)_{MV}$ $\sigma(x\oplus y)=\sigma(x)\oplus \sigma(y\ominus (x\odot y));$\\
$(4)_{MV}$ $\sigma(\sigma(x)\oplus \sigma(y))=\sigma(x)\oplus \sigma(y)$\\
for any $x,y$ in $M.$ The operator $\sigma$ is said to be a {\it
state-MV-operator.}
\end{Def}

We recall  that if $(M,\oplus,\odot,^-,0,1)$ is an MV-algebra, then
$(M,\wedge,\vee,\odot,\to, 0,1),$ where $x\rightarrow y  :=
x^-\oplus y,$ is a BL-algebra satisfying the identity $x^{--}=x.$
Conversely, if $(M,\wedge,\vee,\odot,\to, 0,1)$ is a BL-algebra with
the additional identity $x^{--}=x$, then $(M,\oplus,\odot,^-,0,1)$
is an MV-algebra, where $x\oplus y:= (x^-\odot y^-)^-$ and
$x^-:=x\to 0.$

We recall that the following operations hold in any MV-algebra $M$:\\
$x\ominus y=(x^-\oplus y)^-$ and $x\odot y=(x^-\oplus y^-)^-$ for
any $x,y\in M.$

In what follows, we show that if a BL-algebra is termwise equivalent
to an MV-algebra, then a state-operator $\sigma$ on $M$ taken in the
BL-setup coincides with the notion of a  state-MV-operator in the
MV-setup given by Flaminio and Montagna in \cite{FlMo0,FlMo}, and
vice-versa.

\begin{prop}\label{pr:MV} Let $M$ be an MV-algebra. Then a mapping
$ \sigma:M \to M$ is a state-MV-operator on $M$ if and only if
$\sigma$ is a state-operator on $M$ taken as a BL-algebra. In
addition, in such a case, $\sigma$ is always a strong
state-operator.
\end{prop}

\begin{proof} Let $\sigma$ be a state-operator on $M$ taken as a
BL-algebra. We recall that then $x^{--}=x$ for each $x \in M.$

Axiom $(1)_{MV}$ is property $(a)$ in Lemma \ref{3new}, and
$(2)_{MV}$ is $(b)$ in the same Lemma.

We prove that axiom $(3)_{BL}$ together with the condition
$x^{--}=x$ gives axiom $(3)_{MV}$. First note that in a BL-algebra
satisfying $x^{--}=x$ we have:
\begin{eqnarray*}
x^-\odot (x\oplus y)&=&x^-\odot (x^{--}\oplus y)=x^-\odot (x^-\rightarrow y)=x^-\wedge y=y\wedge x^-\\
&=&y\odot (y\rightarrow x^-)=y\odot (x^-\oplus y^-)=y\odot (x\odot y)^-,
\end{eqnarray*}
using $x\rightarrow y=x^-\oplus y$
(since $x\rightarrow y=x\rightarrow y^{--}=
(x\odot y^-)^-=(x^{--}\odot y^-)^-=x^-\oplus y$)).\\
So
$$\sigma(x^-\odot (x\oplus y))=\sigma(y\odot (x\odot
y)^-).\eqno(*)
$$
Now replacing $x,y$ by $x^-,y^-$ respectively in
axiom $(3)_{BL}$, we get $\sigma(x^-\odot y^-)=\sigma(x^-)\odot
\sigma(x^-\rightarrow x^-\odot y^-)$ which becomes after negation:
$\sigma(x\oplus y)=\sigma(x)\oplus  \sigma(x^-\rightarrow x^-\odot
y^-)^-$ Using $(*)$ we get
\begin{eqnarray*}
\sigma(x\oplus y)&=&\sigma(x)\oplus \sigma(x^-\rightarrow(x\oplus
y)^-)^-=\sigma(x)\oplus \sigma(x^{--}\oplus (x\oplus
y)^-)^-\\&=&\sigma(x)\oplus \sigma(x^-\odot (x\oplus
y))=\sigma(x)\oplus (y\odot (x\odot y)^-)=\sigma (x)\oplus
\sigma(y\ominus (x\odot y)).
\end{eqnarray*}
And thus we obtain axiom $(3)_{MV}$.

We now prove that  axiom $(4)_{BL}$ is in fact axiom $(4)_{MV}$.

Replacing $x,y$ by $x^-, y^-$ respectively in axiom $(4)_{BL}$ we
obtain: $\sigma(\sigma(x^-)\odot \sigma(y^-))=\sigma(x^-)\odot
\sigma(y^-),$ and using $(b),$ we get $\sigma(\sigma(x)^-\odot
\sigma(y)^-)=\sigma(x)^-\odot \sigma(y)^-$. Negating this identity
we obtain $\sigma(\sigma(x)\oplus \sigma(y))=\sigma(x)\oplus
\sigma(y),$ which is axiom $(4)_{MV}.$

Conversely, let $\sigma$ be a state-MV-operator on $M.$ Then:

\noindent $(1)_{BL}:$\ \ $\sigma(0) = 0.$

\noindent $(2)_{BL}:$ \ \ $\sigma(x\rightarrow y) = \sigma(x^-\oplus
y) = \sigma(x^-)\oplus \sigma(y\ominus x^-\odot y) = \sigma(x^-)
\oplus \sigma(y\odot (x^-\odot y^-)^-) = \sigma(x)^-\oplus
\sigma(y\odot (x\oplus y)^-) = \sigma(x)^- \oplus \sigma(y\wedge x)
= \sigma(x)\rightarrow \sigma(x\wedge y).$

\noindent $(3)_{BL}:$\ \ $\sigma(x\odot y) = \sigma((x^-\oplus
y^-)^-) = [\sigma(x^-\oplus y^-)]^- = [\sigma(x^-)\oplus
\sigma(y^-\ominus x^-\odot y^-)]^- = [\sigma(x)^- \oplus
\sigma(y^-\odot (x^-\odot y^-)^-)]^- = [\sigma(x)^-\oplus
\sigma(y\oplus x^-\odot y^-)^-]^- = \sigma(x)\odot \sigma(y\oplus
x^-\odot y^-) = \sigma(x)\odot \sigma(y\vee x^-) = \sigma(x)\odot
\sigma(x^-\oplus x\odot y) = \sigma(x)\odot \sigma(x\rightarrow
x\odot y).$

\noindent $(4)_{BL}:$\ \ $\sigma(\sigma(x)\odot \sigma(y)) =
\sigma((\sigma(x)^-\oplus \sigma(y)^-)^-) = \sigma(\sigma(x^-)\oplus
\sigma(y^-))^- = (\sigma(x^-)\oplus \sigma(y^-))^- =
(\sigma(x)^-\oplus \sigma(y)^-)^- = \sigma(x)\odot \sigma(y).$

\noindent $(5)_{BL}:$\ \ $\sigma(\sigma(x)\rightarrow \sigma(y)) =
\sigma(\sigma(x)^-\oplus \sigma(y)) = \sigma(\sigma(x^-)\oplus
\sigma (y)) = \sigma(x^-)\oplus \sigma(y) = \sigma(x)^-\oplus
\sigma(y) = \sigma(x)\rightarrow \sigma(y)$.

Finally, due to Remark \ref{re:3.6} and $(3')_{BL},$ we see that
$\sigma$ is always a strong state-operator.
\end{proof}

We recall that if $M$ is an MV-algebra and if $x \in M,$ we define
$0\cdot x =0,$ $1\cdot x = x,$ and if $n\cdot x \le x^-,$ we set
$(n+1)\cdot x = (n\cdot x)\oplus x.$  Similarly, if $a \le b^-,$ we
set $a+b:= a\oplus b.$

Therefore, if $\sigma$ is a state-MV-operator on $M$ and  $a+b$ is
defined in $M,$ then $\sigma(a)+\sigma(b)$ is also defined and
$\sigma(a+b)=\sigma(a)+\sigma(b).$  Similarly, $\sigma(n\cdot x) =
n\cdot \sigma(x).$

\begin{Def}
A {\it morphism-state-operator} on a BL-algebra $A$ is a mapping
$\sigma:A\to A$ satisfying $(1)_{BL},(2)_{BL}, (4)_{BL}, (5)_{BL}$
and \\
$(6)_{BL}$ $\sigma(x\odot y)=\sigma(x)\odot \sigma(y)$ for any
$x,y\in A.$\\
The couple $(A, \sigma)$ is called a {\it state-morphism
BL-algebra}.
\end{Def}

We introduce also an additional property:\\
$(7)_{BL}$ $\sigma(x\to y)=\sigma(x)\to \sigma(y)$ for any $x,y\in
A.$

If $A$ is an MV-algebra, then  $(6)_{BL}$ and $(7)_{BL}$ are
equivalent.  In addition, the property ``a state-operator $\sigma$
satisfies  $(7)_{BL}$" is equivalent to the property $\sigma$ is an
endomorphism of $A$ such that $\sigma^2=\sigma,$ see Lemma
\ref{arrow}(3).



\begin{rem}
The state-operators defined in Examples \ref{ex:1}--\ref{ex:2} and
Example \ref{4elemEx} are state-morphism-operators that are also
endomorphisms. Additional examples are in the next section.
\end{rem}

\begin{prop}\label{pr:3.16} Every state-morphism BL-algebra
$(A,\sigma)$ is a strong state BL-algebra, but  the converse is not
true, in general.
\end{prop}

\begin{proof} Since $\sigma$ preserves $\odot,$ due to the identity
holding in any BL-algebra, $x\odot y= x\odot (x^-\vee y),$ $x,y \in
A,$ we have $\sigma$ is a strong state-operator on $A.$

Due to Proposition \ref{pr:MV}, if $A$ is an MV-algebra, then a
strong state-operator on $A$ is  a state-MV-operator and vice-versa.
In \cite[Thm 3.2]{DiDV1}, there are examples of state-MV-operators
that are not state-morphism-operators.
\end{proof}

\begin{cor}
Any linearly ordered strong state BL-algebra $(A,\sigma)$ is a
state-morphism BL-algebra.
\end{cor}

\begin{proof}
It follows from Lemma \ref{linear}(2).
\end{proof}

The latter result will be strengthened in Proposition \ref{pr:3.23}
for any state-operator on a linearly ordered BL-algebra proving that
then it is an endomorphism.

\begin{rem} Let $\sigma$ be a state-operator on a BL-algebra $A$.  (i) If $\sigma$
preserves $\to$, then $\sigma$ is a state-morphism-operator, Lemma
\ref{arrow}(3), (ii) every state-morphism-operator is always a strong
state-operator, Proposition \ref{pr:3.16},  and (iii) every strong
state-operator is a state-operator, Proposition \ref{pr:3.8}.
\end{rem}

\begin{opprob} (1) Does there exists a state-operator that is not strong
?

(2) Does any state-morphism-operator preserve $\to$ ?
\end{opprob}

Some partial answers are presented in the next section.

\section{Examples of State-Operators}

In the present section, we describe some examples of BL-algebras
when every state-operator is even an endomorphism.

First, we describe some state-operators on a finite BL-algebra.

We recall that an element $a$ of a BL-algebra $A$ is said to be {\it
idempotent} if $a \odot a = a.$  Let $\mathrm{Id}(A)$ be the set of
idempotents of $A.$  If $a$ is idempotent, then \cite[Prop
3.1]{DvKo} (i) $x\odot a = a \wedge x $ for any $x \in A,$ (ii) the
filter $F(a)$ generated by $a$ is the set $F(a) = [a,1]:=\{z\in M:\
a\le z \le 1\}.$ In addition, (iii)  $a\in A$ is idempotent iff
$F(a)=[a,1],$ (iv) $\mbox{Id}(A)$ is a subalgebra of $A$ \cite[Cor
3.6]{DvKo}, and (v) if $A$ is finite, then every filter $F$ of $A$
is of the form $F = F(a)$ for some idempotent $a \in A.$

Let $n\ge 1$ be an integer, we denote by
$S_n=\Gamma(\frac{1}{n}\mathbb Z,1)=\{0,1/n,\ldots,n/n\}$  an
MV-algebra. If we set $x_i=i/n,$ then $x_i\odot x_j = x_{(i+j-n)\vee
0}$ and $x_i\to x_j = x_{(n-i+j)\wedge n}$ for $i,j=0,1,\ldots, n.$
If $A$ is a finite MV-algebra, then due to \cite{Cig}, $A$ is a
direct product of finitely many chains, say $S_{n_1},\ldots,
S_{n_k}.$  By \cite{DDL}, every state-MV-operator on a finite
MV-algebra preserves $\odot$ and $\to.$

Let us recall the notion of an ordinal sum of BL-algebras. For two
BL-algebras $A_1$ and $A_2$ with $A_1\cap A_2=\{1\}$, we set $A =
A_1\cup A_2$. On $A$ we define the operations $\odot$ and $\to$  as
follows

$$
\begin{array}{ll}
x \odot y & =\left\{\begin{array}{ll} x \odot_i y & \quad
\mbox{if}\quad x,y\in A_i,\ i=1,2,  \\
 x& \quad \mbox{if}\quad x\in A_1\setminus\{1\},\ y\in A_2,\\
y &\quad \mbox{if}\quad  x \in A_2,\ y \in A_1\setminus\{1\},
\end{array}
\right.
\\
\\
\\
x \rightarrow y  &=\left\{\begin{array}{ll} x \rightarrow_i y &
\quad
\mbox{if}\quad x,y\in A_i, \ i=1,2, \\
 y& \quad \mbox{if}\quad   x\in A_2, \ y \in A_1 \setminus\{1\},\ \\
1& \quad \mbox{if}\quad x \in A_1 \setminus\{1\},\ y \in A_2.\
\end{array}
\right.
\end{array}
$$
Then $A$ is a BL-algebra iff $A_1$ is linearly ordered, and we
denote the ordinal sum $A= A_1 \oplus A_2.$  We can easily extend
the ordinal sum for finitely many summands. In addition, we can do
also the ordinal sum of an infinite system $\{A_i\mid i\in I\}$ of
BL-algebras, where $I$ is a totally ordered set with the least
element $0$ and the last $1$ (to preserve  the prelinearity
condition (3) of Definition \ref{def:BL},  all $A_i$ for $i<1$ have
to be linear BL-algebras).

In \cite{DiLe}, it was shown that every finite BL-algebra is a
direct product of finitely many comets. We note that a {\it comet}
is a finite BL-algebra $A$ of the form $A= A_1\oplus A_2,$ where
$A_1$ is a finite BL-chain, i.e. an ordinal sum of finitely many
MV-chains $S_{n_1},\ldots, S_{n_k},$ and $A_2$ is a finite
MV-algebra.

Let $0_1$ be the least element of $A_1$ and, for any  $x \in A_1,$
we set $x^*= x\to 0_1.$

\begin{lemma}\label{le:3.18}
Let $A = S_n\oplus A_1,$ where  $A_1$ is an arbitrary  BL-algebra,
where $n \ge 1.$  If $\sigma$ is a state-operator on $A$, then
$\sigma(x_i)=x_i,$ where $x_i \in S_n$ for any $i=0,1,\ldots, n,$
and $\sigma$ maps $A_1$ to $A_1.$

Let $\sigma_A$ be the restriction of $\sigma$ onto $A_1.$  Then
\begin{enumerate}

\item[{\rm (i)}]
$\sigma_A(1)=1,$ $\sigma_A(0_1)\le \sigma_A(x)$ for any $x \in A_1;$

\item[{\rm (ii)}]  $\sigma_A(x^*) = \sigma_A(x)\to \sigma_A(0_1);$

\item[{\rm (iii)}] $\sigma_A(0_1)$ is idempotent;

\item[{\rm (iv)}] all conditions $(2)_{BL}-(5)_{BL}$ are true for
$\sigma_A.$
\end{enumerate}
Conversely, if $\sigma_A$ is a mapping from $A_1$ into $A_1$ such
that it satisfies {\rm (iii)}--{\rm(iv)}, then the mapping $\sigma:
A\to A$ defined by $\sigma(x) = \sigma_A(x)$ if $x \in A_1$ and
$\sigma(x)=x$ if $x \in S_n$ is a state-operator on $A.$

\end{lemma}

\begin{proof}  First we show that $\sigma(S_n)\subseteq S_n.$  If
not, then there is $x_i \in S_n$ such that $\sigma(x_n)\notin S_n$
and whence $1>\sigma(x_i)\in A_1.$  Check $\sigma(\sigma(x_i)\to
x_i)=\sigma(x_i)<1$ as well as due to Lemma \ref{3new}$(g)$, we have
$\sigma(\sigma(x_i)\to x_i) = \sigma(x_i) \to \sigma(x_i)=1,$
absurd.

By Proposition \ref{pr:MV}, the restriction of $\sigma$ to $S_n$ is
a state-MV-operator on the MV-algebra $S_n.$ Because $x_i = i\cdot
x_1$ for any $i=1,\ldots,n,$ we have $\sigma(x_i)=i\cdot
\sigma(x_1)$ proving $\sigma(x_1)= x_1$ and $\sigma(x_i) = x_i.$

If $x \in A_1$ and $\sigma(x)=1,$ then trivially $\sigma(x)\in A_1.$
We assert also that $\sigma(A_1)\subseteq A_1.$ Suppose the
converse.  Let now $x \in A_1\setminus S_n$ such that
$\sigma(x)\notin A_1.$ Then $0<\sigma(x)=x_i$ for some
$i=1,\ldots,n-1.$ Then $0=\sigma(x^-) = \sigma(x)^-= x_i^-=x_{n-i}
>0,$ absurd.

Let $\sigma_A$ be the restriction of $\sigma$ onto $A_1.$ Then (i)
and (ii) are evident.

Hence, $\sigma_A(0_1)\odot \sigma_A(0_1)\le \sigma_A (0_1\odot
0_1)=\sigma_A(0_1)$ due to Proposition \ref{3new}(d). But
$\sigma_A(A_1)$ is closed under $\odot$. Therefore,
$\sigma_A(0_1)\le \sigma_A(0_1)\odot\sigma_A(0_1)$ getting
$\sigma_A(0_1)$ is idempotent.

The rest of the proof is now straightforward.
\end{proof}

\begin{lemma}\label{le:3.19} Let $A= S_n \oplus A_1,$ where $A_1=
S_{n_1}\times \cdots \times S_{n_k}.$ Then there are at least $2^k$
state-operators on $A$ such that each of them preserves $\odot$ and
$\to.$
\end{lemma}

\begin{proof}
By Lemma \ref{le:3.18}, every state-operator on the BL-algebra $A$
is on $S_n$ the identity. To have a state-operator on $A$, it is
enough to define it only on $A_1,$ because on the rest of $A_1$ it
is the identity, and this we will do. We have exactly $2^k$
idempotents of $A_1$ and hence $2^k$ different filters of $A_1.$

(1) If $\sigma|A_1 = \mbox{id}_{A_1},$ then $\Ker(\sigma)=\{1\}.$ If
$\sigma|A_1 =\mbox{id}_{\{1\}},$ then $\Ker(\sigma)=A_1$ and both
are state-operators on $A$ that preserve $\odot$ and $\to$.

(2)  Let $J\subseteq \{n_1,\ldots,n_k\}.$ Define $\sigma_J: A_1\to
A_1$ by $\sigma_J(x_1,\ldots,x_k)=(y_1,\ldots,y_k),$ where $y_i=n_i$
if $i\in J$ otherwise $y_i=x_i.$  Then $\sigma_J$ satisfies all
conditions (i)--(iv) of Lemma \ref{le:3.18}, so that $\sigma_J$
defines a state-operator on $A$ that preserves $\odot$ and $\to.$

In addition, every idempotent of $A_1$ is of the form $a=a_J,$ where
$a_J=(x_1,\ldots,x_k)$ with $x_i=n_i$ if $i \in J$ and $x_i=0$
otherwise, whence $a^*_J=a_{J^c}.$   Then
$\sigma_J(x)=1=(n_1,\ldots,n_k)$ iff $x=(x_1,\ldots,x_k)$ with
$x_i=n_i$ if $i\notin J,$ so that $\Ker(\sigma_J)=[a_{J^c},1].$

Hence, there is at least $2^k$ different state-operators $A$ each of
them preserves $\odot$ and $\to.$
\end{proof}

We note that the second   case in (1) in the proof of the latter
lemma is a special case of (2) when $J=\{n_1,\ldots,n_k\},$ and the
first one in (1) is also a special case of (3) when $J = \emptyset.$

Let $A= A_0\oplus A_1$ and let $a \in A_1$ be idempotent. We define
$\sigma_a$ on $A$ as follows: $\sigma_a(x)= 1$ if $a\le x \le 1,$
$\sigma_a(x) = 0_1$ if $0_1\le x\le a^*$ and $\sigma_a(x)=x$
otherwise.

\begin{lemma}\label{le:3.20} Let $A= S_n \oplus A_1,$ where $A_1=
S_{n_1}\times \cdots \times S_{n_k}.$ Let $a$ be  idempotent of
$A_1$ such that  $[a,1]\cup [0_1,a^*]=A_1,$ then $\sigma_a$ is a
state-operator that preserves $\odot$ and $\to.$
\end{lemma}

\begin{proof}
Let $a$ be  idempotent of $A_1$  and $0_1:=(0,\ldots,0)<a<1.$ We set
$x^*:= x\to 0_1$ and  we define $\sigma_a$ on $A$ as follows:
$\sigma_a(x)= 1$ if $a\le x \le 1,$ $\sigma_a(x) = 0_1$ if $0_1\le
x\le a^*$ and $\sigma_a(x)=x$ otherwise.

We claim that  $\sigma_a$ is a state-operator that preserves $\odot$
and $\to$ whenever  $[a,1]\cup [0_1,a^*]=A_1.$  It is enough to
verify the following conditions.

(i) Let $x,y \in [a,1].$ Then $\sigma_a(x\odot y)=1=\sigma_a(x)\odot
\sigma_a(y)$ and $\sigma_a(x\to y)=1=\sigma_a(x)\to \sigma_a(y).$

(ii) Let $0_1\le x,y\le a^*.$ Then $x\to y \ge x\to 0_1 = x^*\ge a$
and thus $\sigma_a(x\to y)=1,$ and $\sigma_a(x)\to \sigma_a(y)
=0_1\to 0_1 =1.$ $\sigma_a(x\odot y)=0_1= \sigma_a(x)\odot
\sigma_a(y).$

(iii) Let $a\le x\le 1$ and $0_1\le y \le a^*.$ Then by Proposition
\ref{4}(6), we have $x\to y \le a\to a^* = a\to (a\to 0_1) = a^2 \to
0_1 = a\to 0_1 = a^*$ that gives $\sigma_a(x\to y) = 0_1$ and
$\sigma_a(x)\to \sigma_a(y) = 1\to 0_1 = 0_1.$

On the other hand, $y\to x \ge a^*\to a = a^* \to (a^* \to 0_1) =
a^* \odot a^* \to 0_1 = a.$ Hence, $\sigma_a(x\to y) = 1$ and
$\sigma_a(x)\to \sigma_a(y) = 0_1 \to 1 =1.$

$\sigma_a(x\odot y) \le \sigma_a(y) =0_1$ and $\sigma_a(x)\odot
\sigma_a(y)=0_1.$
\end{proof}

\begin{rem}\label{re:3.21}
The condition  $[a,1]\cup [0_1,a^*]=A_1$ is satisfied e.g.: if (i)
$a=1$ and $a^*=0_1$ (then $\sigma_a= \sigma_{\{1,\ldots,k\}}$);

\noindent (ii) $A_1 = S_n\times S_1\times \cdots \times S_1,$
$a=(0,\ldots,0,1)$ and $a^*=(n,1,\ldots,1,0).$    Then $\sigma_a \ne
\sigma_J$ for any $J\subseteq\{1,\ldots,k\};$

\noindent (iii) $A_1=\{0,1,2\}\times \{0,1,2\} \times \{0,1\}$ and
$a =(0,0,1),$ $a^*=(2,2,0).$ Then $\sigma_a\ne \sigma_J.$

$\sigma_a$ is not necessarily a state-operator on $A$ if the
condition $[a,1]\cup [0_1,a^*]=A_1$ is not satisfied. Indeed, let
$A_1=\{0,1,2,3,4\}\times \{0,1,2,3,4\}$ and $a=(0,4), a^*=(4,0).$
Then $x=(3,1) \in A_1\setminus([a,1]\cup [(0,0),a^*])$ and $x\odot x
= (3,1)\odot(3,1) = (2,0) \in [(0,0),a^*]$ but $\sigma(x\odot
x)=\sigma(2,0) = (0,0)\ne (2,0)=\sigma(x)\odot \sigma(x)$ that
contradicts $(d)$ of Proposition \ref{3new}.
\end{rem}

In order to prove that every state-operator on a linearly ordered
BL-algebra preserves $\odot$ and $\to,$ we introduce hoops and
Wajsberg hoops.  A {\it hoop} is an algebra $(A,\to,\odot,1)$ of
type $\langle 2,2,1\rangle$ such that for all $x,y,z \in A$ we have
(i) $x\to x=1,$ (ii) $x\odot (x\to y) = y\odot (y\to x),$ and (iii)
$x \to (y\to z)= (x\odot y)\to z.$  Then $\le$ defined by $x \le y$
iff $x\to y=1$ is a partial order and $x\odot(x\to y)= x\wedge y.$ A
{\it Wajsberg hoop} is a hoop $A$ such that $(x\to y)\to y= (y\to
x)\to x,$ $x,y \in A.$  If a Wajsberg hoop has a least element $0,$
then $(A,\odot,\to,0,1)$ is term equivalent to an MV-algebra. For
example, if $G$ is an $\ell$-group written additively with the zero
element $0=0_G$, then the negative cone $G^-=\{g \in G\mid g \le
0\}$ is an unbounded Wajsberg hoop with the greatest element
$1=0_G,$ $g\rightarrow h:=h-(g\vee h)= (h-g)\wedge 0_G,$ and $g\odot
h = g+h,$ $g,h \in G^-.$ If $u$ is a strong unit for $G$, that is,
given $g\in G,$ there is an integer $n\ge 1$ such that $g \le nu,$
we endow the interval $[-u,0_G]:=\{g \in G^-\mid -u \le g \le 0_G\}$
with $x\rightarrow y:=(y-x)\wedge 0_G,$ and $x\odot y = (x+y)\vee
(-u),$ then $[-u,0_G]$ is a bounded Wajsberg hoop with the least
element $0=-u$ and the greatest element $1=0_G.$

Conversely, if $A$ is a Wajsberg hoop, then there is an $\ell$-group
$G$ such that $A$ can be embedded into $G^-$ or onto $[-u,0_G],$ see
e.g. \cite[Prop 3.7]{Dvu}.

If we have a system of hoops, $\{A_i\mid i \in I\},$ where $I$ is a
linearly ordered set, and $A_i \cap A_j=\{1\}$ for $i,j \in I,$
$i\le j$, then we can define an ordinal sum, $A=\bigoplus_{\in
I}A_i$ in a similar manner as for BL-algebras. In such a case, $A$
is always a hoop.

An important result of \cite{AgMo} says that any linearly ordered
BL-algebra $A$ is an ordinal sum of linearly ordered Wajsberg hoops,
$A=\bigoplus_{i\in I}A_i$, where $I$ is a linearly ordered set with
the least element $0$ and $A_0$ is a bounded Wajsberg hoop.

\begin{prop}\label{pr:3.23}  Every state-operator $\sigma$ on a linearly
ordered BL-algebra $A$ preserves both $\odot$ and $\to,$ and it is
an endomorphism such that $\sigma^2 = \sigma.$
\end{prop}

\begin{proof}  Suppose that $\sigma$ is a state-operator on a
linearly ordered BL-algebra $A.$  Due to the Aglian\`o-Montagna
theorem, $A= \bigoplus_{i\in I}A_i.$

It is clear that if $x \in A_i$ is such that $\sigma(x)=1,$ then
trivially $\sigma(x)\in A_i.$ We assert that $\sigma(A_i)\subseteq
A_i$ for any $i \in I.$    Suppose the converse, then there is an
element $x\in A_i\setminus\{1\}$ such that $\sigma(x)\notin A_i,$
whence $\sigma(x)<1,$ and let $\sigma(x) \in A_j$ for $i\ne j.$
There are two cases (i) $i<j$, hence $ x<\sigma(x)$ and
$\sigma(\sigma(x)\to x)= \sigma(x)<1$ as well as
$\sigma(\sigma(x)\to x)=\sigma(x)\to \sigma(x)=1$ taking into
account that in the linear case $\sigma$ preserves $\to.$

(ii) $j<i$ and $\sigma(x)<x$, so that $\sigma(x\to
\sigma(x))=\sigma(x)$ as well as $\sigma(x\to
\sigma(x))=\sigma(x)\to \sigma(x)=1.$

In both case we have  a contradiction, therefore,
$\sigma(A_i)\subseteq A_i$ for any $i\in I.$

Let now $x\in A_i$ and $y \in A_j$ and $i<j.$  If $x =1$ or $y=1,$
then clearly  $\sigma(x\odot y)=\sigma(x)\odot \sigma(y).$  Suppose
$x<1$ and $y<1.$  Then $\sigma(x)\in A_i$ and $\sigma(y)\in A_j$ and
hence $\sigma(x\odot y) = \sigma(x)= \sigma(x)\odot \sigma(y).$

Assume now $x,y \in A_i.$ Then $\sigma(x),\sigma(y) \in A_i.$  If
$x=1$ or $y=1,$ then $\sigma(x\odot y)=\sigma(x)\odot \sigma(y).$
Thus let $x<1$ and $y<1.$

If $i=0$, then $\sigma$ on $A_0$ is a state-MV-operator and
therefore, by Proposition \ref{pr:MV}, $\sigma$ is strong on $A_0$
and because $A_0$ is linear, by Lemma \ref{linear}, $\sigma$
preserves $\odot$ on $A_0.$

Let now $i>0$ and let $A_i$ be bounded, i.e., there exists a least
element $0_i$ in $A_i.$ Then $\sigma(0_i)\le \sigma(x)$ for any $x
\in A_i.$ Due to property $(4)_{BL}$ of Definition \ref{2'}, we have
$\sigma(\sigma(0_i)\odot \sigma(0_i)) =
\sigma(0_i)\odot\sigma(0_i)$, therefore,
$\sigma(0_i)\odot\sigma(0_i) \le \sigma(0_i) \le
\sigma(0_i)\odot\sigma(0_i)$ i.e., $\sigma(0_i)$ is idempotent. But
in the linear $A_i$ there are only two idempotents, $1$ and $0_i.$
Hence either $\sigma(0_i)=1$ or $\sigma(0_i)=0_i.$  In the first
case, $\sigma(x)=1$ for any $x\in A_i$ and the second case, $\sigma$
on $A_i$ is a state-MV-operator, so in both cases, $\sigma$
preserves $\odot$ on $A_i.$

Finally, assume $A_i$ is unbounded. Therefore, $A_i \cong G^-$ for
some linearly ordered $\ell$-group $G,$  and for all $x,y \in A_i,$
we have $x\to (x\odot y) = x \to (x+y) = (x+y)-x = y.$ Hence for
$(3)_{BL},$ we have $\sigma(x\odot y) = \sigma(x)\odot \sigma(x \to
x\odot y)= \sigma(x)\odot \sigma(y).$

From all possible cases we conclude that $\sigma$ preserves $\odot$
on $A$ and by Lemma \ref{linear}, $\sigma$ preserves also $\to,$
i.e. $\sigma$ is an endomorphism such that $\sigma^2 = \sigma.$
\end{proof}

\begin{cor}\label{co:3.24} There is a one-to-one correspondence
between state-operators on a linearly ordered BL-algebra $A$ and
endomorphisms $\sigma: A \to A$ such that $\sigma^2 = \sigma.$
\end{cor}

\begin{rem}\label{re:3.23} In the same way as in the proof of
Proposition \ref{pr:3.23}, we can show that if $\sigma$ is a
state-operator on a BL-algebra $A$ that is an ordinal sum,
$A=\bigoplus_{i\in I} A_i,$ of hoops, then $\sigma(A_i)\subseteq
A_i$ for any $i\in I.$

In addition, suppose a BL-algebra $A = \bigoplus_{i\in I}A_i,$ where
each $A_i$ is a hoop, and $A_0$ is a linear BL-algebra. Let
$\sigma_i:A_i \to A_i$ be a mapping such that conditions  $(1)_{BL}
- (5)_{BL}$ of Definition \ref{2'} are satisfied and if $0_i$ is the
least element of $A_i,$ then $\sigma(0_i)$ is idempotent. Then the
mapping $\sigma:A \to A$ defined by $\sigma(x) = \sigma_i(x)$ if $x
\in A_i,$ is a state-operator on $A.$
\end{rem}

We recall that a {\it t-norm} is a function $t:[0,1]\times [0,1]\to
[0,1]$ such that (i) $t$ is commutative, associative, (ii)
$t(x,1)=x,$ $x \in [0,1],$ and (iii) $t$ is nondecreasing in both
components. If $t$ is continuous, we define $x\odot_t y=t(x,y)$ and
$x\to_t y = \sup\{z\in [0,1]\mid t(z,x)\le y\}$ for $x,y \in [0,1],$
then $\mathbb I_t:=([0,1],\min,\max,\odot_t,\to_t,0,1)$ is a
BL-algebra. Moreover, according to \cite[Thm 5.2]{CEGT}, the variety
of all BL-algebras is generated by all $\mathbb I_t$ with a
continuous t-norm $t.$

There are three important continuous t-norms on $[0,1]$ (i) \L
ukasiewicz: $\mbox{\L}(x,y)=\max\{x+y-1,0\}$ with $x\to_{\mbox
{\tiny \L}} y = \min\{x+y-1,1\},$ (ii) G\"odel: $G(x,y)=\min\{x,y\}$ and
$x\to_Gy = 1$ if $x\le y$ otherwise $x\to_Gy= y,$ and (iii) product:
$P(x,y)=xy$ and $x\to_P y= 1$ if $x\le y$ and $x\to_Py = y/ x$
otherwise. The basic result on continuous t-norms \cite{MoSh} says
that a $t$-norm is continuous iff it is isomorphic to an ordinal sum
where summands are the \L ukasiewicz, G\"odel or product t-norm.

In what follows, we describe all state-operators with respect to
these basic continuous t-norms.

\begin{lemma}\label{le:3.22} $(1)$ If $\sigma$ is a state-operator on
the BL-algebra $\mathbb I_{\mbox{\tiny \L}},$ then $\sigma(x)=x.$

$(2)$ Let $a\in [0,1],$ we set $\sigma_a(x):= x$ if $x \le a$ and
$\sigma_a(x)=1$ otherwise and for $a \in (0,1]$ let $\sigma^a(x)= x$
if $x <a$ and $\sigma^a(x)=1$ otherwise. Then $\sigma_a$ and
$\sigma^a$ are state-morphism-operators on $\mathbb I_G$ preserving
$\to,$ and if $\sigma$ is any state-operator on $\mathbb I_G,$ then
$\sigma = \sigma_a$ or $\sigma= \sigma^a$ for some $a \in [0,1].$

$(3)$ If $\sigma$ is a state-operator on $\mathbb I_P,$ then
$\sigma(x)=x$ for any $x\in [0,1]$ or $\sigma(x) =1 $ for any $x>0.$
\end{lemma}

\begin{proof} (1) If $\sigma$ is a state-operator, then due to
Proposition \ref{pr:MV}, $\sigma$ is a state-MV-operator, so that
$\sigma(n\cdot 1/n)= n\cdot \sigma(1/n)$ so that $\sigma(n/m) = n/m$
and hence $\sigma(x) = x$ for any $x \in [0,1].$

(2) It is straightforward to verify that $\sigma_a$  and $\sigma^a$
are  state-operators on $\mathbb I_G$ that preserves $\odot$ and
$\to.$

Let now $\sigma$ be a state-operator on $\mathbb I_G$ and let
$0<a<1.$ We claim that then $\sigma(a)= 1$ or $\sigma(a)=a.$ (i)
Assume $\sigma(a)<a.$ Then, for any $ z\in [\sigma(a), a],$ we have
$\sigma(z)= \sigma(a)$. Then $\sigma(a\to \sigma(a)) = \sigma(a)$
and $\sigma(a)\to \sigma(a\wedge \sigma(a)) = \sigma(a)\to \sigma(a)
=1,$ a contradiction. (ii) Suppose now $a< \sigma(a)<1,$ then again
$z \in [a,\sigma(a)]$ implies $\sigma(z)=\sigma(a)$ and  we obtain
the same contradiction as in (i). Therefore, if $x\le a < y,$ then
$\sigma(x)= x$ and $\sigma(y)=1.$ If $a_0 = \sup\{a<1\mid
\sigma(a)=a\},$ then $\sigma = \sigma_{a_0}$ or
$\sigma=\sigma^{a_0}.$

(3)  Let $\sigma$ be a state-operator on $\mathbb I_P.$ Due to
Proposition \ref{pr:3.23}, $\sigma$ is an endomorphism. If
$\Ker(\sigma)=\{1\},$ by Lemma \ref{3new}, $\sigma$ is the identity
on $[0,1].$

It is easy to verify that the operator $\sigma_0$ such that
$\sigma_0(0)= 0$ and $\sigma_0(x)=1$ for $x>0$ is a state-operator
on $\mathbb I_P.$  Assume now that for some $x_0<1$ we have
$\sigma(x_0)=1.$ Let $x$ be an arbitrary element such that
$0<x<x_0.$  Then there is an integer $n\ge 1$ such that $x_0^n \le
x.$ Then $\sigma(x_0^n) \le \sigma(x),$ but $\sigma(x_0^n) =
\sigma(x_0)^n = 1 \le \sigma(x)$ proving that $\sigma(x)= 1$ for any
$x>0.$ Hence, $\sigma = \sigma_0.$
\end{proof}

\begin{prop}\label{pr:Go}  Let a BL-algebra $A$ belong to the variety
of G\"odel BL-algebras, i.e., it satisfies the identity $x= x^2.$
Then every state-operator on $A$ is an endomorphism.
\end{prop}

\begin{proof}
Let $x,y \in A$. Since they are idempotent, so are $\sigma(x)$ and
$\sigma(y)$, and we have $\sigma(x\wedge y) = \sigma(x\odot y) \ge
\sigma(x) \odot \sigma(y) = \sigma(x)\wedge \sigma(y)$ when we have
used Lemma \ref{3new}$(d).$  On the other hand, $\sigma(x\wedge y)
\le \sigma(x)\wedge \sigma(y) = \sigma(x)\odot \sigma(y).$ Hence,
$\sigma(x\wedge y)=\sigma(x\odot y) = \sigma(x)\odot \sigma(y) =
\sigma(x)\wedge \sigma(y).$  In view of (1) of Lemma \ref{arrow},
$\sigma$ preserves also $\to,$ so that $\sigma$ is an endomorphism.
\end{proof}

\begin{ex}\label{ex:G1}
 Let $A$ be a finite linear G\"odel BL-algebra. For $a\in A,$ we
put $\sigma_a(x):= x$ if $x \le a$ and $\sigma_a(x)=1$ otherwise and
for $a \in  A\setminus \{0\}$ let $\sigma^a(x)= x$ if $x <a$ and
$\sigma^a(x)=1$ otherwise. Then $\sigma_a$ and $\sigma^a$ are
state-morphism-operators on $A$ preserving $\to,$ and if $\sigma$ is
any state-operator on $A,$ then $\sigma = \sigma_a$ or $\sigma=
\sigma^a$ for some $a \in A.$
\end{ex}

\begin{proof} The proof follows the same ideas as that of (2) of
Lemma \ref{le:3.22}.
\end{proof}

\begin{ex}\label{ex:G2}  If $A$ is an arbitrary linear G\"odel
BL-algebra and for each $a\in A$ we define $\sigma_a$ and $\sigma^a$
in same manner as in Example \ref{ex:G1}, then each of them is an
endomorphism.  But not every state-operator on infinite $A$ is of
such a form.

For example, let $A_Q$ be the set of all rational numbers
of the real interval $[0,1],$ let $a$ be an irrational number from
$[0,1],$  and we define $\sigma_a$ and $\sigma^a,$ then
$\sigma_a=\sigma^a$ but $a \notin A_Q.$  On the other hand, every state-operator $\sigma $ on $A_Q$ is the restriction of some state-operator on $\mathbb I_G$ to $A_Q.$ Indeed, let $a_0 = \sup\{a<1\mid \sigma(a)=a\}.$ If $a_0$ is rational, then $\sigma = \sigma_{a_0}$ or $\sigma = \sigma^{a_0}.$ If $a_0$ is irrational, then $\sigma= \sigma_{a_0}$ and so every $\sigma$ is the restriction of a state-operator on $\mathbb I_G$ to $A_Q.$
\end{ex}

We recall that the variety of MV-algebras, $\mathcal {MV},$ in the
variety of BL-algebras is characterized by the identity $x^{--}=x,$
 the variety of product BL-algebras, $\mathcal P,$ is
characterized by identities $x\wedge x^- =0$ and $z^{--} \to
((x\odot x \to y\odot z)\to (x\to y))=1,$ and the variety of G\"odel
BL-algebras, $\mathcal G,$ is characterized by the identity $x^2=x.$
Due to \cite[Thm 6]{DEGM} or \cite[Lem3]{CEGT}, the variety
$\mathcal {MV}\vee {\mathcal P}$ is characterized by the identity
$x\to (x\odot y) = x^-\vee y.$ Therefore, every state-operator on $A
\in \mathcal {MV}\vee {\mathcal P}$ is strong.

\begin{prop}\label{pr:4.8} If a BL-algebra $A$ is locally finite,
then the identity is a unique  state-operator on $A.$
\end{prop}

\begin{proof}  Assume that for $0<x<1$ we have $\sigma(x)=1.$  Then
there is an integer $n\ge 1$ such that $x^n=0$ and hence
$0=\sigma(0)=\sigma(x^n)\ge \sigma(x)^n=1,$ absurd. Therefore,
$\sigma$ is faithful.

Due to Lemma \ref{100}, $A$ is linear, and therefore, $\sigma$ is
the identity as it follows from $(r)$ of Proposition \ref{3new}.
\end{proof}

We note that every locally finite BL-algebra is in fact an MV-chain,
see e.g. \cite[Thm 2.17]{DiGeIo}.

\begin{rem}\label{re:ex}  Due to Theorem \ref{13}, a state $s$ on a
BL-algebra $A$ is a state-morphism iff $\Ker(s)$ is a maximal
filter. This is not true for state-operators: There are
state-morphism BL-algebras  $(A,\sigma)$ such that $\Ker(\sigma)$ is
not necessarily  a maximal (state-) filter (for the definition of a state-filter, see the beginning of the next section). Indeed, the identity
operator on $\mathbb I_G$ and $\mathbb I_P$ are state-morphism-operators but its kernel is not a maximal (state-) filter.
\end{rem}

\section{State-Filters and  Congruences of State BL-algebras}\label{sec5}

In this section, we concentrate ourselves to filters, state-filters,
maximal state-filters,  congruences on state BL-algebras, and
relationships between them.  In contrast to BL-algebras when every
subdirectly irreducible state BL-algebra is  linearly ordered,  for
the variety of state BL-algebras, this is not necessarily the case.
However, the image of such a subdirectly irreducible state
BL-algebra is always linearly ordered.

\begin{Def}
Let $(A, \sigma)$ be a state BL-algebra (or a state-morphism
BL-algebra). A nonempty set $F\subseteq A$ is called a {\it
state-filter} (or a {\it state-morphism filter}) of $A$ if $F$ is a
filter of $A$ such that if $x\in F,$ then $\sigma(x)\in F.$ A proper
state-filter of $A$ that is not contained as a proper subset in any
other proper filter of $A$ is said to be {\it maximal}. We denote by
$\Rad_\sigma(A)$ the intersection of all maximal state-filters of
$(A,\sigma).$
\end{Def}

For example, $\Ker(\sigma):=\{x\in A \mid \sigma(x)=1\}$ is a
state-filter of $(A,\sigma).$

We recall that there is a one-to-one relationship  between
congruences and state-filters on a state BL-algebra $(A,\sigma)$ as
follows.  If $F$ is a state-filter, then the relation $\sim_F$ given
by $x \sim_F y$ iff $x\to y,y\to x\in F$ is a congruence of the
BL-algebra $A$ and due to Lemma \ref{3new}(h) $\sim_F$ is also a
congruence of the state BL-algebra $(A,\sigma).$

Conversely, let $\sim$ be a congruence of $(A,\sigma)$ and set
$F_\sim:=\{x\in A \mid x\sim 1\}.$ Then $F_\sim$ is a state-filter
of $(A,\sigma)$ and $\sim_{F_\sim} = \sim$ and $F= F_{\sim_F}.$

It is known that any subdirectly irreducible BL-algebra is linear.
This is not true in general for state BL-algebras.

\begin{ex}\label{ex:5.1}  Let $B$ be a simple BL-algebra, i.e. it
has  only two filters.  Set $A=B\times B$ and let
$\sigma_1(a,b)=(a,a)$ and $\sigma_2(a,b)=(b,a)$ for $(a,b)\in A.$
Then  $\sigma_1$ and $\sigma_2$ are state-morphisms that are also
endomorphisms and $\Ker(\sigma_1)= B\times \{1\},$
$\Ker(\sigma_2)=\{1\}\times B.$ In addition, $(A,\sigma_1)$ and
$(A,\sigma_2)$ are isomorphic subdirectly irreducible state-morphism
BL-algebras that are not linear; $\Ker(\sigma_1)$ and
$\Ker(\sigma_2)$ are the least nontrivial state-filters.
\end{ex}

A little bit more general is the following example:

\begin{ex}\label{ex:5.2}  Let $B$ be a linear BL-algebra, $C$ a
nontrivial subdirectly irreducible BL-algebra with the smallest
nontrivial filter $F_C,$ and let $h:B\to C$ be a BL-homomorphism. On
$A=B\times C$ we define $\sigma_h: A\to A$ by
$$ \sigma(b,c):=(b,h(b)),\quad (b,c)\in B\times C,\eqno(5.1)
$$
Then $(A,\sigma_h)$ is a subdirectly irreducible state-morphism
BL-algebra that is not linearly ordered,  $\Ker(\sigma_h)
=\{1\}\times C$ and $F=\{1\}\times F_C$ is the smallest nontrivial
state-filter of $A.$
\end{ex}

\begin{prop}\label{pr:2.6}  Let $(A,\sigma)$ be a state BL-algebra
and  $X\subseteq A.$ Then the state-filter $F_\sigma(X)$ generated
by $X$ is the set

$$F_\sigma(X)=\{x\in A \mid x\ge (x_1\odot
\sigma(x_1))^{n_1}\odot\cdots \odot (x_k\odot \sigma(x_k))^{n_k}, x_i
\in X, n_i\ge 1, k\ge 1\}.$$

If $F$ is a  state-filter of $A$ and $a \not\in A,$ then the
state-filter of $A$ generated by $F$ and $a$ is the set
$$ F_\sigma(F,a)=\{x \in A: \ x \ge i \odot (a\odot \sigma(a))^n,\
i \in F,\ n \ge 1\}.$$

A proper state-filter $F$ is a maximal state-filter if and only if,
for any $a \not\in F,$ there is an integer $n\ge 1$ such that
$(\sigma(a)^n)^- \in F.$
\end{prop}

\begin{proof} The first two parts are evident.

Now suppose  $F$ is maximal, and let $a \not\in F$. Then
$F_\sigma(F,a) = A$ and there are $i \in F$ and an integer $n\ge 1$
such that $0 = i \odot (a \odot \sigma(a))^n.$ Applying $\sigma$ to
this equality, we have $0 = \sigma(0)\ge \sigma(i) \odot
\sigma(a)^{2n}$. Therefore, $ \sigma(i)\le (\sigma(a)^{2n})^-  \in
F.$

Conversely, let $a$ satisfy the condition.  Then the element
$0=(\sigma(a)^n)^- \odot \sigma(a)^n \in F_\sigma(F,a)$ so that $F$
is maximal.
\end{proof}

Nevertheless a subdirectly irreducible state BL-algebra $(A,\sigma)$
is not necessarily linearly ordered, while $\sigma(A)$ is always linearly
ordered:

\begin{theo}\label{th:5.4} If $(A,\sigma)$ is a subdirectly
irreducible state BL-algebra, then  $\sigma(A)$ is linearly ordered.

If $\Ker(\sigma)=\{1\},$ then $(A,\sigma)$ is subdirectly
irreducible if and only if $\sigma(A)$ is a subdirectly irreducible
BL-algebra.
\end{theo}

\begin{proof}  Let $F$ be the least nontrivial state-filter of
$(A,\sigma).$ By the minimality of $F$, $F$ is generated by some
element $a<1.$ Hence, $F=\{x\in A\mid x \ge (a\odot \sigma(a))^n,\ n
\ge 1\}.$  Assume that there are two elements $\sigma(x),\sigma(y)
\in \sigma(A)$ such that $\sigma(x)\not\le \sigma(y)$ and
$\sigma(y)\not\le \sigma(x).$ Then $\sigma(x)\to \sigma(y)<1$ and
$\sigma(y)\to \sigma(x) <1$ and let $F_1$ and $F_2$ be the
state-filters generated by $\sigma(x)\to \sigma(y)$ and
$\sigma(y)\to \sigma(x),$ respectively. They are nontrivial,
therefore they contain $F,$ and $a \in F_1\cap F_2.$  Because
$\sigma$ on $\sigma(A)$ is the identity, by Proposition
\ref{pr:2.6}, there is an integer $n \ge 1$ such that $a \ge
(\sigma(x)\to \sigma(y))^n$ and $a \ge (\sigma(y)\to \sigma(x))^n.$
By \cite[Cor 3.17]{DGI1}, $a\ge (\sigma(x)\to \sigma(y))^n \vee
(\sigma(y)\to \sigma(x))^n = 1$ and this is a contradiction.

Since $\sigma$ on $\sigma(A)$ is the identity, then every
state-filter of $\sigma(A)$ is a BL-filter and vice-versa.

Assume $\Ker(\sigma)=\{1\}$ and  let $F$ be the least nontrivial
state-filter of $(A,\sigma).$ Then $F_0:=F\cap \sigma(A)$ is a
nontrivial state-filter of $\sigma(A).$ We assert that $F_0$ is the
least nontrivial filter of $\sigma(A).$ Indeed, let $I$ be another
nontrivial filter of $\sigma(A)$ and let $F'$ be the state-filter of
$(A,\sigma)$ generated by $I.$ Then $F' \supseteq F$ and $I= F'\cap
\sigma(A) \supseteq F \cap \sigma(A)=F_0$ proving that $F_0$ is the
least state-filter.

Conversely, assume that $J$ is the least nontrivial filter of
$\sigma(A).$ Let $F(J)$ be the state-filter of $(A,\sigma)$
generated by $J.$ We assert that $F(J)$ is the least nontrivial
state-filter of $(A,\sigma).$ Let $F$ be any nontrivial state-filter
of $(A,\sigma).$ Then $F \cap \sigma(A)$ is a nontrivial filter of
$\sigma(A)$, hence $F\cap \sigma(A) \supseteq J$ which proves $F
\supseteq F(J).$
\end{proof}


\begin{prop}\label{30} If $\sigma$ is a state-operator on a
BL-algebra $A$, then
$$\Rad(\sigma(A)) \subseteq \sigma(\Rad(A)).\eqno(5.2)
$$  If,
in addition, $\sigma$ is a strong state-operator on $A$, then
$$\sigma(\Rad(A))=\Rad(\sigma(A)). \eqno(5.3)$$
\end{prop}

\begin{proof}
Let $\sigma$ be a state-operator on $A$ and choose $y\in
\Rad(\sigma(A)).$ Since $(y^n)^-\leq y $ for all integers $n\geq 1,$
then $y\in \Rad(A).$ But $y=\sigma(y),$ so $y\in \sigma(\Rad(A)).$
Thus $\Rad(\sigma(A))\subseteq \sigma(\Rad(A)).$

Let now $\sigma$ be a strong state-operator on $A.$ Let $x\in
\Rad(A),$ then from Proposition \ref{27''} we have $(x^n)^-\leq x$
for any $n\in \N.$ As $\sigma$ is increasing, we get
$\sigma(x^n)^-\leq \sigma(x).$ Using condition $(3')_{BL}$ we obtain
$\sigma(x\odot x)=\sigma(x)\odot \sigma(x^- \vee x) = \sigma(x)\odot
\sigma(x)$ and by induction, we have,
$\sigma(x^{n+1})=\sigma(x)\odot \sigma(x^-\vee x)\odot \cdots \odot
\sigma((x^n)^- \vee x)= \sigma(x)^{n+1}.$ Hence,
$\sigma(x^n)=\sigma(x)^n$ and thus $(\sigma(x)^n)^-\leq \sigma(x),$
for any $n\in \N$. Therefore, $\sigma(x)\in \Rad(\sigma(A)),$ and
thus $\sigma(\Rad(A))\subseteq \Rad(\sigma(A)).$
\end{proof}

\begin{prop}\label{pr:2.7}  Let $F$ be a maximal state-filter
of a state BL-algebra $(A,\sigma)$ and let, for some $a \in A,$
$\sigma(a)/F$ be co-infinitesimal in $A/F.$ Then $\sigma(a) \in F.$
\end{prop}

\begin{proof}  Let $(\sigma(a)^n)^-/F \le \sigma(a)/F$ for each $n
\ge 1.$  If $\sigma(a) \notin F,$  by Proposition \ref{pr:2.6},
there is an integer $n\ge 1$ such that $(\sigma(a)^n)^- \in F$.
Hence $(\sigma(a)^n)^-/F =  1/F = \sigma(a)/F,$  so that $\sigma(a)
\in F$ that is a contradiction. Therefore, $\sigma(a) \in F.$
\end{proof}

\begin{prop}\label{pr:2.8} $(1)$ Let $N$ be a state BL-subalgebra
of a state BL-algebra $(A,\sigma)$. If $J$ is a maximal state-filter
of $A$, so is $I=J\cap N$ in $N.$  Conversely, if $I$ is a maximal
state-filter of $N,$ there is a maximal state-filter $J$ of $A$ such
that $I=J\cap N.$

$(2)$ If $I$ is a state-filter  (maximal state-filter) of
$(A,\sigma)$, then $\sigma(I)$ is a filter (maximal filter) of
$\sigma(A)$ and $\sigma(I) = I \cap \sigma(A).$

$(3)$ If $I$ is a (maximal) filter of $\sigma(A)$, then
$\sigma^{-1}(I)$ is a (maximal) state-filter of $A$ and
$\sigma^{-1}(I) \cap \sigma(A)= I.$

\end{prop}

\begin{proof} $(1)$ Since $J\cap N$ is a state-filter of $N$,
the first half of the first statement follows from Proposition
\ref{pr:2.6}.  Conversely, let $I$ be a maximal state-filter of $N$
and let $F(I)=\{x\in A\mid x\ge i$ for some $i \in I\}$ be the
state-filter of $(A,\sigma)$ generated by $I.$ Since $0\notin F(I),$
choose a maximal state-filter $J$ of $A$ containing $F(I)$.  Then
$J\cap N \supseteq F(I)\cap N = I.$  The maximality of $I$ entails
$J\cap N = I.$

$(2)$ Let $I$ be a state-filter  of  $(A,\sigma).$   An easy
calculation shows that $\sigma(I) = I\cap \sigma(A),$ therefore by
(1), $\sigma(I)$ is a filter of $\sigma(A).$

Let now $I$ be maximal and suppose now that $\sigma(a) \not\in
\sigma(I).$ Therefore, $a \not\in I$ and there is an integer such
that $(\sigma(a)^n)^- \in I$ and hence, $(\sigma(a)^n)^- \in
\sigma(I),$ so that $\sigma(I)$ is a maximal filter in $\sigma(A).$

$(3)$ Finally, using the basic properties of $\sigma$, we have that
$\sigma^{-1}(I)$ is a filter  of $A.$  Suppose now $y \in
\sigma^{-1}(I),$ then $\sigma(y) \in I$ and $\sigma(y) =
\sigma(\sigma(y)) \in \sigma^{-1}(I)$ proving $\sigma^{-1}(I)$ is a
state-filter of $(A,\sigma).$

Now, suppose $I$ is a maximal filter of $\sigma(A)$,  then $0\notin
\sigma^{-1}(I),$ and let $x \not\in \sigma^{-1}(I).$ Hence
$\sigma(x) \not\in I$ and the maximality of $I$  implies that there
is an integer $n \ge 1$ such that $(\sigma(x)^n)^- \in I$. Then
$\sigma((\sigma(x)^n)^-) = (\sigma(x)^n)^- \in I,$ and
$(\sigma(x)^n)^- \in \sigma^{-1}(I).$ By Proposition \ref{pr:2.6},
$\sigma^{-1}(I)$ is a maximal state-filter of $(A,\sigma).$
\end{proof}

\begin{prop}\label{pr:2.9}  Let $(A,\sigma)$ be a state
BL-algebra. Then
$$ \sigma(\mbox{\rm Rad}(A)) \supseteq \mbox{\rm Rad}(\sigma(A)) =
\sigma(\mbox{\rm Rad}_\sigma(A)).$$

\end{prop}

\begin{proof} The left-side inclusion follows from (5.2).

Suppose $x \in \mbox{Rad}_\sigma(A).$ Then $x \in I$ for every
maximal state-filter $I$ on $(A,\sigma)$
 and $\sigma(x) \in \sigma(I) = I \cap
\sigma(A)$ and $\sigma(I)$ is a maximal filter of $\sigma(A)$ by
Proposition \ref{pr:2.8}. If now $J$ is a maximal filter of
$\sigma(A),$ by Proposition \ref{pr:2.8}, $\sigma^{-1}(J)$ is a
maximal state-filter of $(A, \sigma(A))$ and $\sigma(x) \in
\sigma(\sigma^{-1}(J))= J$. Therefore, $\sigma(x) \in
\mbox{Rad}(\sigma(A)),$ and $\sigma(\mbox{Rad}_\sigma(A)) \subseteq
\mbox{Rad}(\sigma(A)).$


Conversely, let $x \in \mbox{Rad}(\sigma(A)).$ Then $x \in I$ for
every maximal filter $I$ on $\sigma(A).$  Hence, if $\sigma(y) = x$,
then $y \in \sigma^{-1}(I)$ and $\sigma^{-1}(I)$ is a maximal
state-filter of $(A,\sigma).$ Suppose that $y \in J$ for each
maximal state-filter of $(A,\sigma).$ Then $x \in \sigma(J) = J \cap
\sigma(A)$ and $\sigma(J)$ is a maximal filter of  $\sigma(A)$ and
$y \in \sigma^{-1}(\sigma(J)) = J.$ Therefore, $y \in
\mbox{Rad}_\sigma(A)$ and $x \in \sigma(\mbox{Rad}_\sigma(A)),$
giving $\mbox{Rad}(\sigma(A)) \subseteq
\sigma(\mbox{Rad}_\sigma(A)).$
\end{proof}

\begin{opprob} (1) Describe subdirectly irreducible elements of
state-morphism BL-algebras.

(2) Does there exist a state BL-algebra such that in (5.2) we have a
proper inclusion~?
\end{opprob}

\section{States on State BL-algebras}

In the present section, we give  relations between states on
BL-algebras and state-operators. We show that starting from a state
BL-algebra $(A,\sigma)$ we can always define a state on $A$:  Take a
maximal filter $F$ of the BL-algebra $\sigma(A).$  Then the quotient
$\sigma(A)/F$ is isomorphic to a subalgebra of the standard
MV-algebra of the real interval $[0,1]$ such that the mapping
$s:\sigma(a) \mapsto \sigma(a)/F$ is an extremal state on
$\sigma(A).$ And this can define a state on $A$ as it is shown in
the next two statements. Some reverse process in a nearer sense is
also possible, i.e. starting with a state $s$ on $A$, we can define
a state-operator on the tensor product of $[0,1]$ and an appropriate
MV-algebra connected with $s$ and $A,$ see the last remark of this
section.

As it was already shown before, due to \cite{DvRa}, the notions of a
Bosbach state and a Rie\v{c}an state coincide for BL-algebras.

\begin{prop}\label{9}
Let $(A, \sigma)$ be a state BL-algebra, and let $s$ be a  state on
$\sigma(A)$. Then $s_{\sigma}(x):=s(\sigma(x)),$ $x \in A,$ is a
state on $A.$
\end{prop}

\begin{proof} We show that $s_\sigma$ is a Rie\v{c}an state.
Let $x\bot y.$ Then $x^{--}\leq y^-$ and as $\sigma$ is increasing, we obtain $\sigma(x^{--})\leq \sigma(y^-),$ or equivalently, $\sigma(x)^{--}\leq \sigma(y)^-,$ which implies $\sigma(x)\bot \sigma(y).$ And thus $\sigma(x)+\sigma(y)=\sigma(y)^-\rightarrow \sigma(x)^{--}=\sigma(x)^-\rightarrow \sigma(y)^{--}.$\\
We will prove that $\sigma(x+y)=\sigma(x)+\sigma(y)$ when $x\bot y.$\\
We have: $\sigma(x+y)=\sigma(x\oplus y)=\sigma(x^-\rightarrow
y^{--})=\sigma(x)^-\rightarrow \sigma(x^-\wedge y^{--})$ using axiom
$(2)_{BL}.$
Since $x^{--}\leq y^-, $ i.e. $y^{--}\leq x^-,$ we get $\sigma(x\oplus y)=\sigma(x)^-\rightarrow \sigma(y)^{--}=\sigma(x)+\sigma(y).$ Thus we get $\sigma(x+y)=\sigma(x)+\sigma(y).$\\
We check now that $s_{\sigma}$ is a Rie\v can state.\\
$(RS1_{\sigma})$: $s_{\sigma}(x+y)=s(\sigma(x+y))=s(\sigma(x)+\sigma(y))=s_{\sigma}(x)+s_{\sigma}(y)$ for $x\bot y.$\\
$(RS2_{\sigma})$: $s_{\sigma}(0)=s(\sigma(0))=s(0)=0,$ using $(1)_{BL}$ and $RS2.$\\
Thus $s_{\sigma}$ is a Rie\v can state on $A.$
\end{proof}

\begin{prop}\label{15}
Let $(A, \sigma)$ be a state-morphism BL-algebra, and let $s$ be an
extremal state on $\sigma(A).$ Then $s_{\sigma}(a):=s(\sigma(a)),$
$a\in A,$ is an extremal state on $A.$
\end{prop}

\begin{proof}
By Proposition \ref{9}, $s_{\sigma}$ is a state on $A.$ Then
$s_{\sigma}(x\odot y)=s(\sigma(x\odot y))=s(\sigma(x)\odot
\sigma(y))=s(\sigma(x))\odot s(\sigma(y))=s_{\sigma}(x)\odot
s_{\sigma}(y).$ Using Theorem \ref{13}, $s_{\sigma}$ is an extremal
state on $A.$
\end{proof}

\begin{Def}\label{16}
Let $(A, \sigma)$ be a state BL-algebra and $s$ a  state on $A.$ We
call $s$ a {\it $\sigma$-compatible  state} if and only if
$\sigma(x)=\sigma(y)$ implies $s(x)=s(y)$ for  $x, y \in A.$  We
denote by ${\mathcal S}_{\rm com}(A,\sigma)$ the set of all
$\sigma$-compatible states on $(A,\sigma).$
\end{Def}

In what follows, we show that ${\mathcal S}_{\rm com}(A,\sigma) \ne
\emptyset$ and, in addition, ${\mathcal S}_{\rm com}(A,\sigma)$ is
affinely homeomorphic with ${\mathcal S}(\sigma(A)),$ i.e. the
homeomorphism preserves  convex combinations of states.

\begin{theo}\label{17}
Let $(A, \sigma)$ be a state BL-algebra. Then the set of
$\sigma$-compatible states ${\mathcal S}_{\rm com}(A,\sigma) \ne
\emptyset$ and it is affinely homeomorphic with the set ${\mathcal
S}(\sigma(A))$ of all states on the BL-algebra $\sigma(A).$
\end{theo}

\begin{proof}  We define $\psi:{\mathcal S}(\sigma(A)) \to {\mathcal S}_{\rm
com}(A,\sigma)$  as follows: $\psi(s')(x)=s'(\sigma(x))$ for all $s'
\in \mathcal S(\sigma(A)),$ $x \in A.$ By  Proposition \ref{9},
$\psi(s')=s'_{\sigma}$ is a Rie\v{c}an state on $A.$ The mapping
$\psi(s')$ is a $\sigma$-compatible Rie\v{c}an state on $A$ since if
$\sigma(x)=\sigma(y),$ then
$\psi(s')(x)=s'(\sigma(x))=s'(\sigma(y))=\psi(s')(y).$  Hence,
${\mathcal S}_{\rm com}(A,\sigma) \ne \emptyset.$

Now we define a mapping $\phi: {\mathcal S}_{\rm com}(A,\sigma) \to
{\mathcal S}(\sigma(A))$ by $\phi(s)(\sigma(x)):=s(x)$ for any $x
\in A,$ where $s$ is a $\sigma$-compatible state on $A.$

We first prove that $\phi(s)$  is well defined since if
$\sigma(x)=\sigma(y),$ then by definition $s(x)=s(y),$ so
$\phi(s)(\sigma(x))=\phi(s)(\sigma(y)).$

We prove now condition $(BS1):$ Let $x,y\in \sigma(A),$ so
$x=\sigma(x)$ and $y=\sigma(y)$ according to $(l)$ from Lemma
\ref{3new}. We have:
\begin{eqnarray*}
\phi(s)(x)+\phi(s)(x\rightarrow y)&=&\phi(s)(\sigma(x))+\phi(s)(\sigma(x)\rightarrow \sigma(y))\\
&=&\phi(s)(\sigma(\sigma(x)))+\phi(s)(\sigma(\sigma(x)\rightarrow
\sigma(y)))\\
&=&s(\sigma(x))+s(\sigma(x)\rightarrow \sigma(y))\\
&=&s(\sigma(y))+s(\sigma(y)\rightarrow \sigma(x))\\
&=&\phi(s)(\sigma(\sigma(y)))+\phi(s)(\sigma(\sigma(y)\rightarrow \sigma(x)))\\
&=&\phi(s)(\sigma(y))+\phi(s)(\sigma(y)\rightarrow \sigma(x))
=\phi(s)(y)+\phi(s)(y\rightarrow x).
\end{eqnarray*}
We used axiom $(5)_{BL},$ property $(j)$ of Lemma \ref{3new} and the
fact that $s$ is a state.

Now we check condition $(BS2).$\\
$\phi(s)(0)=\phi(s)(\sigma(0))=s(0)=0$ and similarly it follows that $\phi(s)(1)=1.$\\
Thus $\phi(s)$ is a state on $\sigma(A).$

Finally, $\phi\circ \psi= \rm id_{{\mathcal S}\sigma(A)}$ since
$\phi(\psi(s'))(\sigma(x))=\psi(s')(x)=s'(\sigma(x)).$ Also,
$\psi\circ \phi=\rm id_{{\mathcal S}_{\rm com}(A,\sigma)}$ because
$(\psi\circ \phi)(s)(x)=\phi(s)(\sigma(x))=s(x).$

It is straightforward that both mappings preserve convex
combinations of states, and both are also continuous with respect to
the weak topology of states.
\end{proof}

\begin{cor}\label{co:4.6} Let $(A,\sigma)$ be a state BL-algebra.
Then every $\sigma$-compatible state on $(A,\sigma)$ is  a weak
limit of  a net of convex combinations of extremal
$\sigma$-compatible states on $(A,\sigma)$, and the set of extremal
$\sigma$-compatible states is relatively compact in the weak
topology of states.
\end{cor}

\begin{proof} Due to Theorem \ref{17}, extremal $\sigma$-compatible
states correspond to extremal states on $\sigma(A)$ under some
affine homeomorphism. By the Krein-Mil'man theorem,  \cite[Thm
5.17]{Goo}, every  state on $\sigma(A)$  is a weak limit of  a net
of convex combinations of extremal states on $\sigma(A),$ therefore,
the same is true also for $\sigma$-compatible states on
$(A,\sigma).$ From Theorem \ref{13}, we have that the set of
extremal states on $\sigma(A)$ is relatively compact, therefore, the
same is true for the set of extremal $\sigma$-compatible states on
$(A,\sigma).$
\end{proof}

\begin{rem}\label{re:last}  Let $s$ be a state on a BL-algebra $A.$
Since $\Ker(s)$ is a filter of $A,$ then $A/\Ker(s)$ is an
MV-algebra because $s(a^{--})=s(a)$ and $s(a^{--}\to a)=1$ for any
$a \in A.$ According to \cite{FlMo}, we define the tensor product
$T:=[0,1]\otimes A/\Ker(s)$ in the category of MV-algebras. Then $T$
is generated by elements $\alpha\otimes (a/\Ker(s)),$ where $\alpha
\in [0,1]$ and $a\in A,$ and $A/\Ker(s)$ can be embedded into $T$
via $a/\Ker(s) \mapsto 1\otimes (a/\Ker(s)).$

Let $\mu$ be any state on $A/\Ker(s),$ in particular, $\mu$ can be a
state defined by $a/\Ker(s) \mapsto s(a),$ $(a\in A).$ We define an
operator $\sigma_\mu:T\to T$ by $\sigma_\mu(\alpha\otimes
(a/\Ker(s))) = \alpha \mu((a/\Ker(s)) \otimes (1/\Ker(s)).$ Due to
\cite[Thm 5.3]{FlMo}, $\sigma_\mu$ is always a state-MV-operator on
$T,$ and in view of \cite[Thm 3.1]{DiDV1}, $\sigma_\mu$ is a
state-morphism-operator if and only if $\mu$ is an extremal state.
\end{rem}

\section{Classes of State-Morphism BL-algebras}

We characterize some classes of state-morphism BL-algebras, like
simple state BL-algebras, semisimple state BL-algebras, local state
BL-algebras, and perfect state BL-algebras.

\begin{Def}
A state BL-algebra $(A, \sigma)$ is called {\it simple} if
$\sigma(A)$ is simple. We denote by $\mathcal{SSBL}$ the class of
all simple state BL-algebras.
\end{Def}

\begin{rem}
Let  $(A,\sigma)$ be a state BL-algebra. Note that if $A$ is simple,
then $\sigma(A)$ is simple, thus $(A,\sigma)$ is simple.
\end{rem}

\begin{theo}\label{31}
Let $(A, \sigma)$ be a state-morphism BL-algebra. The following are equivalent:\\
$(1)$ $(A, \sigma)\in \mathcal{SSBL};$\\
$(2)$ $\Ker(\sigma)$ is a maximal filter of $A.$
\end{theo}

\begin{proof}
For the direct implication consider $x\notin \Ker(\sigma),$ i.e.
$\sigma(x)<1.$ Since $\sigma(A)$ is simple,  using Lemma \ref{100},
we have that $\ord(\sigma(x))<\infty.$ This means that there exists
$n\in \N$ such that $(\sigma(x))^n=0.$ By $(6)_{BL}$ and negation,
it follows that $\sigma((x^n)^-)=1,$ that is $(x^n)^-\in
\Ker(\sigma).$ By Proposition \ref{27} we obtain that $\Ker(\sigma)$
is maximal filter of $A.$

 Vice-versa, assume
$\Ker(\sigma)$ is a maximal filter of $A.$ Let $\sigma(x)<1,$ then
$\sigma(x)\notin \Ker(\sigma).$ But $\Ker(\sigma)$ is maximal, so by
Proposition \ref{27} we get that there exists $n\in \N$ such that
$((\sigma(x))^n)^-\in \Ker(\sigma),$ thus
$\sigma(((\sigma(x))^n)^-)=1.$ But $\sigma$ is a
state-morphism-operator, so $\sigma(x^n)=\sigma(x)^n.$ Using Lemma
\ref{3new}$(b)$ and $(j)$ we obtain  $\sigma((x^n)^-)=1.$ Since
$\sigma(x^n)\leq \sigma(x^n)^{--}=1^-=0,$ we get $\sigma(x^n)=0.$
This means $\ord(\sigma(x))<\infty$ for any $\sigma(x)\neq 1,$ which
implies that $\sigma(A)$ is simple, using Lemma \ref{100}.
\end{proof}

\begin{Def}
A state BL-algebra $(A, \sigma)$ is called {\it semisimple} if
$\Rad(\sigma(A))=\{1\}.$ We denote by $\mathcal{SSSBL}$ the class of
all semisimple state BL-algebras.
\end{Def}

\begin{theo}\label{31'}
Let $(A, \sigma)$ be a state-morphism BL-algebra. The following are equivalent:\\
$(1)$ $(A, \sigma)\in \mathcal{SSSBL};$\\
$(2)$ $\Rad(A)\subseteq \Ker(\sigma).$
\end{theo}

\begin{proof}
Assume $(A, \sigma)$ is semisimple, that is $\Rad(\sigma(A))=\{1\}$,
so by Proposition \ref{30}, $\sigma(\Rad(A))=\{1\},$ thus
$\Rad(A)\subseteq \Ker(\sigma).$

Conversely, assume $\Rad(A)\subseteq \Ker(\sigma),$ that means
$\sigma(\Rad(A))=\{1\},$ so $\Rad(\sigma(A))=\{1\},$ thus $(A,
\sigma)$ is semisimple.
\end{proof}

\begin{theo}\label{32}
Let $(A, \sigma)$ be a state BL-algebra. The following are equivalent:\\
$(1)$ $A$ is perfect;\\
$(2)$ $(\forall x\in A,$ $\sigma(x)\in \Rad(A)$ implies $x\in
\Rad(A))$ and $\sigma(A)$ is perfect.
\end{theo}

\begin{proof}
First, let $A$ be a perfect BL-algebra and let $\sigma(x)\in
\Rad(A).$ Assume $x\notin \Rad(A)$, i.e. $x\in \Rad(A)^-,$ so
$x^-\in \Rad(A)$ using Remark \ref{26'}. Then from Corollory
\ref{26''} it follows that $\sigma(x)^-\leq x^-$ and negating it we
get $x^{--}\leq \sigma(x)^{--}.$ Using again Corollory \ref{26''},
we obtain $x^{--}\leq x^-,$ i.e. $\sigma(x)^{--}\leq \sigma(x)^-.$
Thus $\sigma(x)^{--}\leq \sigma(x)^-\leq x^-$ which is a
contradiction, since $\sigma(x)^{--}\in \Rad(A),$ but
$\sigma(x)^-\notin \Rad(A)$ and $\Rad(A)$ is a filter. Thus $x\in
\Rad(A).$
Also, $\sigma(A)$ is perfect, since it is a BL-subalgebra of a perfect BL-algebra.\\
Now we prove the converse implication. Assume $\sigma(A)$ is perfect
and take $x\in A.$ If $\sigma(x)\in \Rad(\sigma(A))\subseteq
\Rad(A),$ then by the hypothesis we get $x\in \Rad(A).$ If
$\sigma(x)\in \Rad(\sigma(A))^-\subseteq \Rad(A)^-,$ then
$\sigma(x)^-=\sigma(x^-)\in \Rad(A),$ and by the hypothesis $x^-\in
\Rad(A),$ and using again Remark \ref{26'}, it follows that $x\in
\Rad(A)^-.$ Thus $A$ is perfect.
\end{proof}

\begin{Def}
Let $(A, \sigma)$ be a state BL-algebra. A state-operator $\sigma$ is called {\it
radical-faithful} if, for every $x\in A,$ $\sigma(x)\in \Rad(A)$
implies $x\in \Rad(A).$
\end{Def}

The first implication of Theorem \ref{32} can be restated in the
following way: Every state-operator on a perfect BL-algebra is
radical-faithful.

\begin{theo}\label{33}
Let $(A, \sigma)$ be a  state-morphism BL-algebra with radical-faithful $\sigma.$ The following are equivalent:\\
$(1)$ $A$ is a local BL-algebra;\\
$(2)$ $\sigma(A)$ is a local BL-algebra.
\end{theo}

\begin{proof}
First, assume $A$ is local. Then according to Proposition
\ref{27'''}, $\ord(x)<\infty$ or $\ord(x^-)<\infty$ for any $x\in
A,$ i.e. there exists $n\in \N$ such that $x^n=0$ or $(x^-)^n=0,$ so
either $\sigma(x)^n=0$ or $(\sigma(x)^-)^n=0,$ which means
$\sigma(A)$ is local.

Now we prove the converse implication. Assume $\sigma(A)$ is local,
then by Proposition \ref{27.1}, $\Rad(\sigma(A))$ is primary. Let
$(x\odot y)^-\in \Rad(A)$; then, since $\sigma$ is a state-morphism,
we get by (5.3): $\sigma((x\odot y)^-) = (\sigma(x)\odot
\sigma(y))^-\in \sigma(\Rad(A))= \Rad(\sigma(A))\subseteq \Rad(A).$
Therefore, $(\sigma(x)^n)^-\in \Rad(\sigma(A))$ or
$(\sigma(y)^n)^-\in \Rad(\sigma(A))$ for some $n$. We can assume
$(\sigma(x)^n)^-\in \Rad(\sigma(A)),$ that is $\sigma((x^n)^-)\in
\Rad(\sigma(A))\subseteq \Rad(A)$, and since $\sigma$ is a
radical-faithful state-morphism-operator, we get $(x^n)^-\in
\Rad(A).$ Similarly, $(\sigma(y)^n)^-\in \Rad(\sigma(A)) ,$  implies
$(y^n)^-\in \Rad(A).$ Since $\Rad(A)$ is proper, we get $\Rad(A)$ is
primary. Using Proposition \ref{0}, $A/\Rad(A)$ is local, so there
exists a unique maximal filter, say $F,$ in $A/\Rad(A).$  Let
$J=\{x\in A \mid x/\Rad(A)\in F\}.$ Then $J$ is a proper filter of
$A$ containing $\Rad(A).$ To show that $J$ is maximal, we use
Proposition \ref{27}. Thus let $x \in A \setminus J.$  Then
$x/\Rad(A) \notin F$ and locality of $A/\Rad(A)$ yields that there
is an integer $n \ge 1$ such that $x^n/\Rad(A)=0/\Rad(A),$ i.e.
$(x^n)^- \in \Rad(A)\subseteq J $ proving $J$ is maximal.

We claim that there is no other maximal filter $I\ne J$ of $A.$ If
not, there is an element $x\in I\setminus J$ and again there is an
integer $n\ge 1$ such that $(x^n)^-\in \Rad(A)\subseteq I.$ This
gives a contradiction $x^n, (x^n)^- \in I.$ Consequently, $A$ is
local.
\end{proof}

\begin{theo}\label{34}
Let $(A, \sigma)$ be a state-morphism BL-algebra with radical-faithful $\sigma.$ The following statements are equivalent:\\
$(1)$ $(A, \sigma)\in \mathcal{SSBL};$\\
$(2)$ $A$ is a local BL-algebra and $\Ker(\sigma)=\Rad(A).$
\end{theo}

\begin{proof}
First, assume $(A, \sigma)\in \mathcal{SSBL}.$ Then $\sigma(A)$ is
simple and local and from Theorem \ref{33}, we have $A$ is local.
Thus $\Ker(\sigma)\subseteq \Rad(A).$ Since $\sigma(A)$ is simple
(and also semisimple),  then $\Rad(\sigma(A)) = \{1\}$ and using
$\sigma(\Rad(A)) = \Rad(\sigma(A)) = \{1\}$, we get
$\Rad(A)\subseteq \Ker(\sigma)$. Thus $\Ker(\sigma)=\Rad(A).$

Now assume $A$ is local and $\Ker(\sigma)=\Rad(A).$  Therefore,
$\Ker(\sigma)$ is a maximal filter in $A$ and Theorem \ref{31}
yields that $\sigma(A)$ is simple.

\end{proof}


\setlength{\parindent}{0pt}

\vspace*{3mm}
\begin{flushright}
\begin{minipage}{148mm}\sc\footnotesize
Lavinia Corina Ciungu\\
Polytechnical University of Bucharest
Splaiul Independen\c tei 113, Bucharest, Romania \&\\
State University of New York - Buffalo, 244 Mathematics Building, Buffalo NY, 14260-2900, USA\\
{\it E--mail address}: {\tt lavinia\underline{
}ciungu@mathem.pub.ro, lcciungu@buffalo.edu}\vspace*{3mm}

Anatolij Dvure\v censkij, Marek Hy\v{c}ko\\
Mathematical Institute, Slovak Academy of Sciences, \v Stef\'anikova 49, SK-81473 Bratislava, Slovakia\\
{\it E--mail address}: {\tt
\{dvurecenskij, hycko\}@mat.savba.sk}\vspace*{3mm}

\end{minipage}
\end{flushright}

\end{document}